\documentclass[11pt,reqno]{amsart}
\usepackage[left=0.9in,top=0.9in,right=0.9in,bottom=0.9in]{geometry}
\usepackage{amsaddr}

\usepackage{amsfonts,amsmath,amsthm,amssymb,latexsym,mathrsfs,stmaryrd}
\usepackage[capbesideposition=right]{floatrow}
\usepackage{hyperref}
\usepackage{mathtools}
\usepackage{graphicx}
\graphicspath{ {images/} }
\usepackage{multirow}
\usepackage{subcaption}
\usepackage{verbatim}
\usepackage{xfrac}

\usepackage{tikz}\usetikzlibrary{cd,fit,matrix,arrows,decorations.pathmorphing}
\tikzset{commutative diagrams/.cd}

\numberwithin{equation}{section}

\newtheorem{theorem}{Theorem}[section]
\newtheorem{corollary}[theorem]{Corollary}
\newtheorem{lemma}[theorem]{Lemma}
\newtheorem{proposition}[theorem]{Proposition}

\theoremstyle{definition}

\newtheorem{definition-theorem}[theorem]{Definition-Theorem}

\theoremstyle{remark}

\newtheorem{remark}{Remark}
\newtheorem*{remark*}{Remark}
\newtheorem*{example*}{Example}

\usepackage{enumitem}
\setenumerate[1]{label=\textup{(\arabic*)}}
\setenumerate[2]{label=\textup{(\roman*)}}

\newcommand\R{{\mathbb R}}
\newcommand\F{{\mathbb F}}

\newcommand\Fq{{\mathbb F}_q}

\newcommand\subeq{\subseteq}

\newcommand\supeq{\supseteq}

\DeclarePairedDelimiter{\abs}{\lvert}{\rvert}

\DeclarePairedDelimiter{\set}{\{}{\}}

\newcommand{\defn}{\textbf}
\newcommand{\AGM}{\mathrm{AGM}}
\newcommand{\agmchain}{(a_0,b_0)\overset{\AGM}{\mapsto} (a_1,b_1)\overset{\AGM}{\mapsto} \dots \overset{\AGM}{\mapsto} (a_n,b_n)}

\newcommand{\qthreemodfour}{q\equiv 3\bmod 4}
\newcommand{\qfivemodeight}{q\equiv 5\bmod 8}
\newcommand{\Sadv}[2]{S_{#2}^{\mathrm{adv}_{#1}}}
\newcommand{\Sback}[2]{S_{#2}^{\mathrm{back}_{#1}}}
\newcommand{\Tadv}[2]{T_{#2}^{\mathrm{adv}_{#1}}}
\newcommand{\Tback}[2]{T_{#2}^{\mathrm{back}_{#1}}}

\begin{document}
\title[Dynamical structure of AGM]{Dynamical structure of AGM over finite fields with order congruent to $5 \bmod 8$}
\author{Daniel A. N. Vargas}
\address{Department of Mathematics, Harvey Mudd College}
\email{dvargas@g.hmc.edu}
\thanks{}


\date{}

\date{\today}

\begin{abstract}
    Motivated by classical works of Gauss and Euler on the AGM, Ono and his collaborators \cite{MR4567422, mcspirit2023hypergeometryagmfinitefields} investigated the union of AGM sequences over finite fields $\Fq$, where $q \equiv 3 \bmod 4$, which they refer to as swarms of jellyfish. A recent preprint \cite{kayath2024agmaquariumsellipticcurves} extends some of their results to all finite fields with odd characteristic. For $q \equiv 5 \bmod 8$, we reveal finer details about the structure of the connected components, which turn out to be variants of jellyfish with longer and branched tentacles. Moreover, we determine the total population of these swarms in terms of the celebrated base ``congruent number" elliptic curve.
\end{abstract}

\maketitle

\section{Introduction and statements of results}
Given two nonnegative numbers $\alpha$ and $\beta$, their arithmetic and geometric means are $\frac{\alpha+\beta}{2}$ and $\sqrt{\alpha\beta}$, respectively. Motivated by Gauss and Euler, the iterated sequence of arithmetic-geometric means $$(\alpha, \beta) \mapsto \left(\frac{\alpha+\beta}{2}, \sqrt{\alpha\beta}\right) \mapsto \left(\frac{\frac{\alpha+\beta}{2}+\sqrt{\alpha\beta}}{2}, \sqrt{\frac{\alpha+\beta}{2} \sqrt{\alpha\beta}}\right) \mapsto \cdots$$ has been well-studied over the real numbers. The two entries of the ordered pairs rapidly approach the same limit \cite{Borwein_book}, called the \defn{arithmetic-geometric mean}. Using this process, Gauss showed that if $a_0 = 1$, $b_0 = \sqrt{2}$, $a_{n+1} := \frac{a_n + b_n}{2}$, and $b_{n+1} := \sqrt{a_n b_n}$, then
$$\lim_{n \to \infty} \frac{{a_n}^2}{1 - \sum_{i=0}^n 2^{i-1} ({a_i}^2 - {b_i}^2)} = \pi$$
with remarkably fast convergence, and he also connected the arithmetic-geometric mean to elliptic integrals \cite{Borwein_book}.

At first glance, the recursive relation $b_{n+1} := \sqrt{a_n b_n}$ presents both an obstacle and a potential ambiguity: If $a_n b_n$ is negative, it does not have any square roots in $\mathbb{R}$, and if $a_n b_n$ is positive, it has two real square roots, making it unclear which to choose. Conveniently, these two issues do not arise: Because $-1$ is not a square in $\mathbb{R}$, we know that exactly one choice of $b_{n+1} = \pm \sqrt{a_n b_n}$ results in $a_{n+1} b_{n+1}$ being a square, and always choosing this value for $b_{n+1}$ makes the arithmetic-geometric mean well-defined over $\mathbb{R}$.

In analogy with the classical study of AGM sequences, recent works have focused on finite field analogues. For finite fields with order congruent to 3 modulo 4, the residue class $-1$ is not a quadratic residue, as in the case of $\mathbb{R}$. In this case, the authors of \cite{MR4567422} show that the disjoint union of all AGM sequences is a directed graph where each graph looks like a ``jellyfish", which will be explained shortly. Here, we extend this to fields $\Fq$ with $q \equiv 1 \bmod 4$, especially focusing on $q \equiv 5 \bmod 8$.

The central object of this paper is the arithmetic-geometric mean (AGM) process over an arbitrary field with odd characteristic. This has been previously investigated in \cite{MR4567422, kayath2024agmaquariumsellipticcurves, mcspirit2023hypergeometryagmfinitefields}, but this paper will provide additional results, with a focus on $\Fq$ for $q \equiv 5 \bmod 8$. Much like \cite{MR4567422} and \cite{mcspirit2023hypergeometryagmfinitefields}, this paper will exclude nodes $(\alpha, \beta)$ where any of $\alpha$, $\beta$, $\alpha + \beta$, and $\alpha - \beta$ are equal to zero because these correspond to degenerate cases.

Specifically, over any field $K$ with characteristic not equal to two (so we can divide by two for the arithmetic mean), we define the set of \defn{nontrivial nodes} by
\begin{equation}S_K:=\set{(\alpha,\beta)\in K^2: \alpha,\beta,\alpha\pm \beta\neq 0},\end{equation}
and define a subset $F_K \subeq {S_K}^2$ by
\begin{equation}F_K:=\set{((\alpha,\beta),(\gamma,\delta))\in {S_K}^2: 2\gamma=\alpha+\beta, \delta^2=\alpha\beta}.\end{equation}
Thus, $F_K$ is the set of parent-child pairs of nontrivial nodes (where the parent is first) and thus a subset of ${S_K}^2$. When viewed as a directed graph, it refers to \defn{AGM advancement}. This directed graph can be viewed as a disjoint union of connected components that represents the AGM process over $K$. If $((\alpha, \beta),(\gamma, \delta))\in F_K,$ we also write
\begin{equation}(\alpha, \beta) \overset{\AGM}{\mapsto} (\gamma, \delta), \qquad (\alpha, \beta) \overset{\AGM_K}{\mapsto} (\gamma, \delta), \qquad (\gamma, \delta) \overset{\AGM_K}{\mapsfrom} (\alpha, \beta), \qquad \text{or simply } (\alpha, \beta) \mapsto (\gamma, \delta)\end{equation}
and say $(\alpha, \beta)$ \defn{advances} to $(\gamma, \delta)$ under the rules of the AGM (sometimes referred to as simply advancing), that $(\alpha, \beta)$ is $(\gamma, \delta)$'s \defn{parent}, or that $(\gamma, \delta)$ is $(\alpha, \beta)$'s \defn{child}.

Recent works of Griffin, McSpirit, Ono, Saikia, and Tsai \cite{MR4567422, mcspirit2023hypergeometryagmfinitefields} considered for the first time AGM over finite fields. They showed that if a prime power $q$ satisfies $q\equiv 3\bmod 4$, then the compilation of all iterated AGM sequences over $\Fq$ form a directed graph whose connected components have a specific structure called ``jellyfish". Moreover, by relating AGM to $2$-isogeny networks of elliptic curves over $\Fq$, they are able to prove statements and formulas that relate the sizes and/or the number of jellyfish to Gauss's or Hurwitz--Kronecker class numbers. Moreover, these elliptic curve isogeny networks are similar to work done by Sutherland in \cite{Volcanoes}, which can be viewed as a generalization of jellyfish but with undirected graphs.

\begin{figure}[h]
    \includegraphics[width=0.5\linewidth]{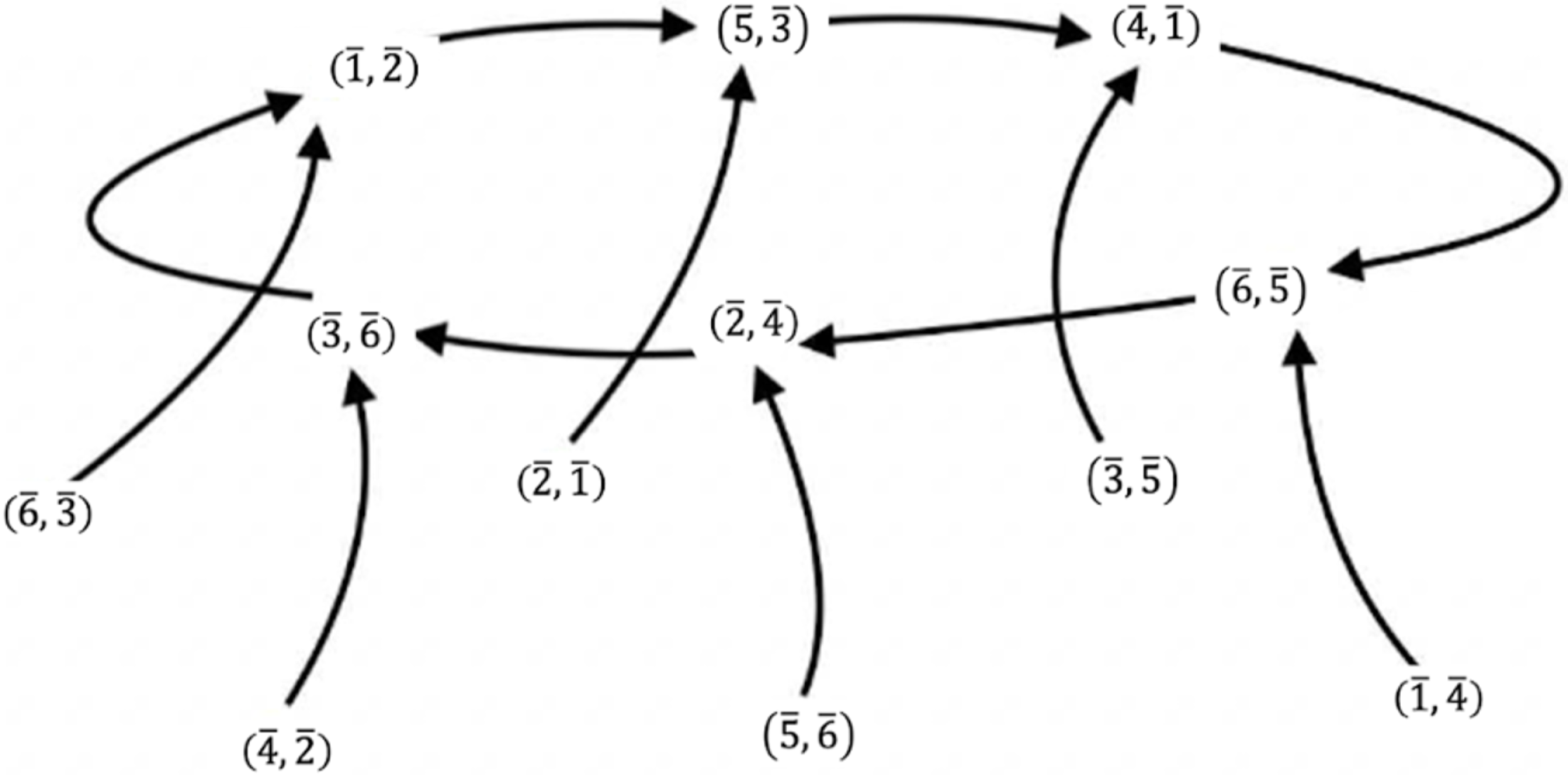} 
    \caption{The AGM over $\mathbb{F}_7$ consists of a single jellyfish.}
    \label{fig:forwards F7 jellyfish}
\end{figure}

An indispensable common feature in both settings above is that one can devise a rule to \emph{canonically} advance from a node to the next, so the starting node determines the entire iterated AGM sequence. Indeed, over $\R$, we always take the square root of $a_n b_n = {b_{n+1}}^2$ such that $b_{n+1}$ has the same sign as $a_{n+1}$.
Over $\Fq$ with $q\equiv 3\bmod 4$, we add an extra starting requirement that $a_0b_0$ is a square in $\Fq$, and during advancement we take $a_{n+1}=\dfrac{a_n+b_n}{2}$ and $b_{n+1}$ to be the unique square root of $a_nb_n$ such that $a_{n+1}b_{n+1}$ is a square; this rule works precisely because $-1$ is not a square in $\Fq$ and because the product of two non-squares is always a square in a finite field. In \cite[\S 5]{mcspirit2023hypergeometryagmfinitefields}, the authors raise the question of whether AGM can be defined on other finite fields. The answer turns out to be yes, but to facilitate this discussion, we first need definitions to distinguish between different types of nodes by how many times they can be advanced under the rules of AGM advancement. For example, over $\F_{29}$, the node $\left(\overline{1}, \overline{11}\right)$ cannot be advanced at all because $\overline{11}$ is not a square in $\F_{29}$, the node $\left(\overline{6}, \overline{25}\right)$ can be advanced once but not twice because $\left(\overline{6}, \overline{25}\right) \overset{\AGM_{\F_{29}}}{\mapsto} \left(\overline{1}, \overline{11}\right)$ and $\left(\overline{6}, \overline{25}\right) \overset{\AGM_{\F_{29}}}{\mapsto} \left(\overline{1}, \overline{18}\right)$ but $\overline{11}$ and $\overline{18}$ are not squares in $\F_{29}$ and $\left(\overline{6}, \overline{25}\right)$ has only two children, and the node $\left(\overline{13}, \overline{28}\right)$ can be advanced twice because $\left(\overline{13}, \overline{28}\right) \overset{\AGM_{\F_{29}}}{\mapsto} \left(\overline{6}, \overline{25}\right) \overset{\AGM_{\F_{29}}}{\mapsto} \left(\overline{1}, \overline{11}\right)$.

Given $n\geq 0$ and a field $K$ with characteristic not equal to two, we say a node $(a_0,b_0)\in S_K$ is \defn{$n$-times advanceable} if there exists an iterated AGM chain of length $n$ starting at $(a_0,b_0)$, namely,
\begin{equation}\agmchain,\end{equation}
in which case $(a_n, b_n)$ is called a \defn{descendant} of $(a_0, b_0)$, and $(a_0, b_0)$ is called an \defn{ancestor} of $(a_0, b_0)$. We say a node $(a_0,b_0)$ is \defn{indefinitely advanceable} if $(a_0,b_0)$ is $n$-times advanceable for all $n$.

The $(a_k, b_k)$s do not need to be all distinct. For example, over $\mathbb{F}_7$, the node $\left(\overline{1}, \overline{4}\right)$ is nine-times advanceable because
\begin{align*}
    \left(\overline{1}, \overline{4}\right) &\overset{\AGM_{\F_7}}{\mapsto} \left(\overline{6}, \overline{5}\right) \overset{\AGM_{\F_7}}{\mapsto} \left(\overline{2}, \overline{4}\right) \overset{\AGM_{\F_7}}{\mapsto} \left(\overline{3}, \overline{6}\right) \overset{\AGM_{\F_7}}{\mapsto} \left(\overline{1}, \overline{2}\right) \overset{\AGM_{\F_7}}{\mapsto} \left(\overline{5}, \overline{3}\right)
    \\\overset{\AGM_{\F_7}}{\mapsto} \left(\overline{4}, \overline{1}\right) &\overset{\AGM_{\F_7}}{\mapsto} \left(\overline{6}, \overline{5}\right) \overset{\AGM_{\F_7}}{\mapsto} \left(\overline{2}, \overline{4}\right) \overset{\AGM_{\F_7}}{\mapsto} \left(\overline{3}, \overline{1}\right)
\end{align*}
$$$$
even though $\left(\overline{6}, \overline{5}\right)$ and $\left(\overline{2}, \overline{4}\right)$ are repeated.

Define $\Sadv{n}{K}\subeq S_K$ to be the set of $n$-times advanceable nodes, and $\Sadv{\infty}{K}\subeq S_K$ to be the set of indefinitely advanceable nodes. Clearly,
$$S_K=\Sadv{0}{K}\supeq \Sadv{1}{K}\supeq \dots\supeq \Sadv{n}{K}\supeq \cdots \supeq \Sadv{\infty}{K}=\bigcap_{n\geq 0} \Sadv{n}{K}.$$
Furthermore, we define $F_K^{\mathrm{adv}_\infty}$ to be the induced subgraph of $F_K$ on the vertex set $\Sadv{\infty}{K}$, namely $F_K$ restricted to $\Sadv{\infty}{K}$.

\begin{example*}
    Using the previously-mentioned nodes, we have that $\left(\overline{1}, \overline{11}\right), \left(\overline{1}, \overline{18}\right) \notin S_{\F_{29}}^{\mathrm{adv}_1}$, $\left(\overline{6}, \overline{25}\right) \in S_{\F_{29}}^{\mathrm{adv}_1} \setminus S_{\F_{29}}^{\mathrm{adv}_2}$, $\left(\overline{13}, \overline{28}\right) \in S_{\F_{29}}^{\mathrm{adv}_2}$, and $\left(\overline{1}, \overline{4}\right) \in S_{\F_7}^{\mathrm{adv}_9}$. In fact, because the provided AGM sequence beginning at $\left(\overline{1}, \overline{4}\right)$ contains repeated nodes, we can conclude that $\left(\overline{1}, \overline{4}\right) \in S_{\F_7}^{\mathrm{adv}_\infty}$ because we can repeat the length-6 cycle arbitrarily many times via
    \begin{align*}
        \left(\overline{1}, \overline{4}\right) &\overset{\AGM_{\F_7}}{\mapsto} \left(\overline{6}, \overline{5}\right) \overset{\AGM_{\F_7}}{\mapsto} \left(\overline{2}, \overline{4}\right) \overset{\AGM_{\F_7}}{\mapsto} \left(\overline{3}, \overline{6}\right) \overset{\AGM_{\F_7}}{\mapsto} \left(\overline{1}, \overline{2}\right) \overset{\AGM_{\F_7}}{\mapsto} \left(\overline{5}, \overline{3}\right) \overset{\AGM_{\F_7}}{\mapsto} \left(\overline{4}, \overline{1}\right)
        \\&\overset{\AGM_{\F_7}}{\mapsto} \left(\overline{6}, \overline{5}\right) \overset{\AGM_{\F_7}}{\mapsto} \left(\overline{2}, \overline{4}\right) \overset{\AGM_{\F_7}}{\mapsto} \left(\overline{3}, \overline{6}\right) \overset{\AGM_{\F_7}}{\mapsto} \left(\overline{1}, \overline{2}\right) \overset{\AGM_{\F_7}}{\mapsto} \left(\overline{5}, \overline{3}\right) \overset{\AGM_{\F_7}}{\mapsto} \left(\overline{4}, \overline{1}\right)
        \\&\overset{\AGM_{\F_7}}{\mapsto} \left(\overline{6}, \overline{5}\right) \overset{\AGM_{\F_7}}{\mapsto} \cdots
    \end{align*}
    so every node in this cycle is in $S_{\F_7}^{\mathrm{adv}_\infty}$ for the same reason. In fact, for finite fields, a node is indefinitely advanceable if and only if it has some descendant in a cycle in $F_{\Fq}$.
\end{example*}

Under this framework, structural results about AGM over $\Fq$ with $\qthreemodfour$ can be succinctly summarized as follows.
\begin{theorem}[\cite{MR4567422}]
    If $q$ is a prime power with $\qthreemodfour$, then
    \begin{enumerate}
        \item (Stabilization) $\Sadv{\infty}{\Fq}=\Sadv{1}{\Fq}.$
        \item $\Sadv{\infty}{\Fq}=\set{(a,b)\in S_K: ab\in (\Fq^\times)^2}.$
        \item $\abs{\Sadv{\infty}{\Fq}}=\frac{(q-1)(q-3)}{2}.$
        \item The $F_{\Fq}$ directed graph restricted\footnote{This means we take the induced subgraph of $F_{\Fq}$ with vertex set $\Sadv{\infty}{\Fq}$.} to $\Sadv{\infty}{\Fq}$ is a \defn{function}, i.e., each node has out-degree $1$.
        \item Each connected component of the $F_{\Fq}$ directed graph restricted to $\Sadv{\infty}{\Fq}$ has a bell-head with length-one tentacles pointing to each of its nodes, constituting a directed analogue of a 2-volcano with depth 1 from \cite{Volcanoes}, as shown in Figure~\ref{fig:forwards F7 jellyfish}.
    \end{enumerate}\label{3mod4results}
\end{theorem}

Note that the properties above are intrinsic properties of the AGM directed graph without reference to the artificial advancement rule; rather, the unique advancement rule \emph{emerges} from the assertion (4). While properties (1)-(4) are trivial in the $\qthreemodfour$ setting thanks to the non-squareness of $-1$, asking for the analogue of \emph{each} of them over other finite fields poses a separate question whose answer remains unclear.

Our first set of results generalizes each of these to $q \equiv 5 \bmod 8$ and requires the arithmetic of the ``congruent number" elliptic curve
\begin{equation}E:\ y^2=x^3-x.\end{equation}
This curve has complex multiplication with endomorphism ring $\mathbb{Z}\!\left[\sqrt{-1}\right]$, which implies \cite{MR1401749} that its \defn{trace of Frobenius} at an odd prime power $p^t$ is
\begin{equation}a_{p^t} = \pi^t + \overline{\pi}^t, \quad \text{where } \pi = \begin{cases}\sqrt{-p} & \text{if } p \equiv 3 \bmod 4 \\ (-1)^{n+m}(2n+1) + 2mi & \text{if } p \equiv 1 \bmod 4,\end{cases}\label{trace of frobenius formula}\end{equation}
where $n$ and $m$ are nonnegative integers satisfying $p = 4m^2 + (2n+1)^2$.

\begin{theorem}\label{5mod8results}
    If $q$ is a prime power with $\qfivemodeight$, then
\begin{enumerate}
    \item (Stabilization) $\Sadv{\infty}{\Fq}=\Sadv{2}{\Fq}.$
    \item $\Sadv{\infty}{\Fq}=\set{(a,b)\in S_K: 4ab(a+b)^2\in (\Fq^\times)^4}$, and is nonempty if and only if $q\neq 5,13.$
    \item $\abs{\Sadv{\infty}{\Fq}}=\frac{(q-1)}{4}(q-7-a_q),$
    where $a_q$ is the Frobenius trace of the elliptic curve $y^2=x^3-x$ over $\Fq$.
    \item Each node in $\Sadv{\infty}{\Fq}$ has out-degree one in $F_{\Fq}^{\mathrm{adv}_\infty}$, corresponding to advancement over $\Sadv{\infty}{\Fq}$ being a single-valued function.
    \item Each connected component of $F_{\Fq}^{\mathrm{adv}_\infty}$ has a bell-head with length-two, $Y$-shaped tentacles pointing to each of its nodes, constituting a directed analogue of a 2-volcano with depth 2 from \cite{Volcanoes}, as shown in Figure~\ref{fig:forwards F29 jellyfish}.
\end{enumerate}
\end{theorem}

\begin{example*}
    We consider one directed graph component for $q = 29$ for Theorem~\ref{5mod8results}.
    \begin{figure}[h!]
        \includegraphics[width=0.5\linewidth]{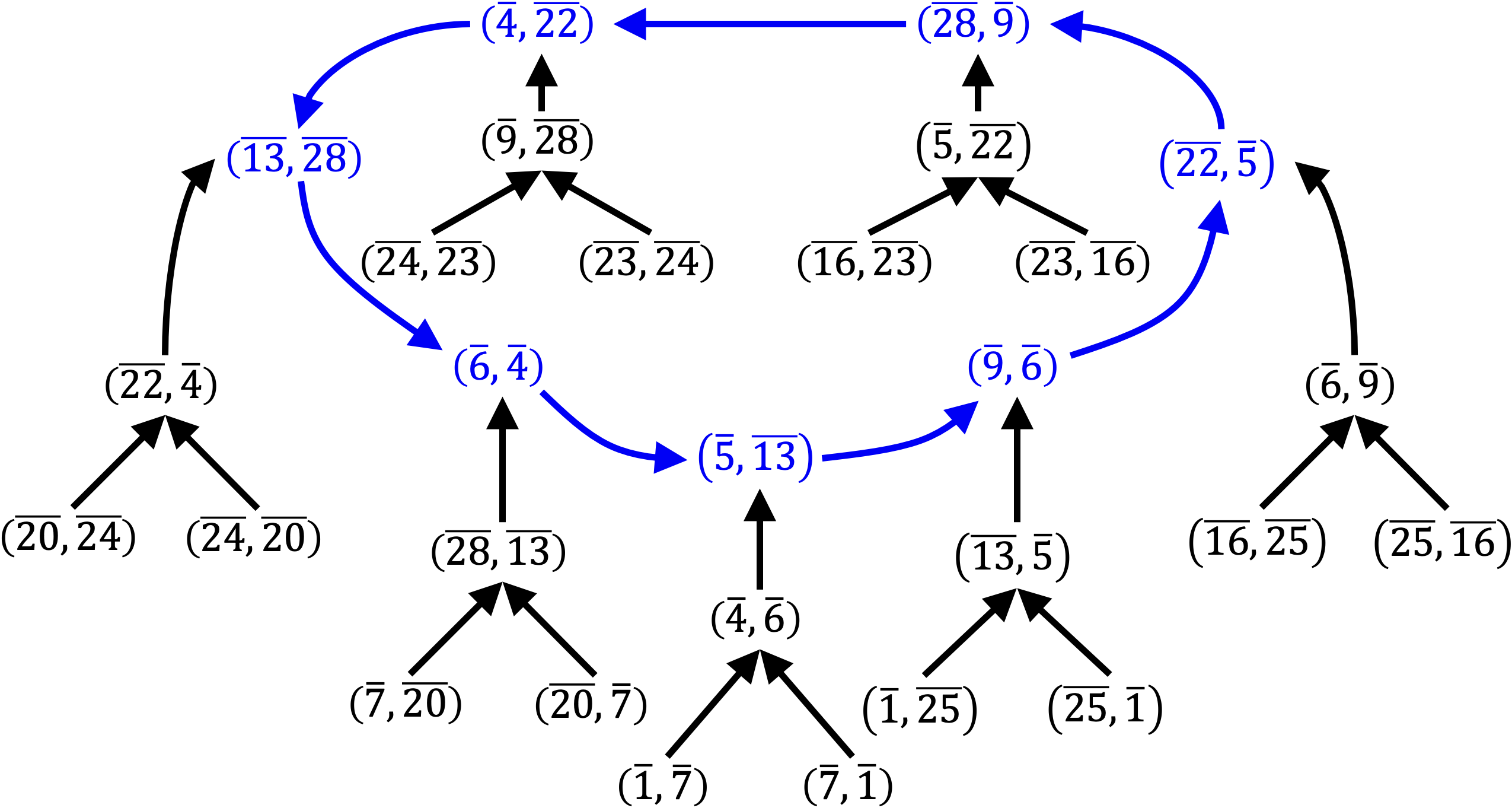}
        \caption{One connected component of $F_{\F_{29}}^{\mathrm{adv}_\infty}$, demonstrating the bell-head cycle and branched length-two tentacles, with the cyclic nodes and cyclic arrows colored blue}
        \label{fig:forwards F29 jellyfish}
    \end{figure}
\end{example*}

One key feature for the proof of Theorem~\ref{5mod8results} is that we are able to determine certain facts about the structure of a jellyfish, such as tentacle length. This feature can be understood using the concept of backtracking in the AGM sequence, which was originally employed in \cite{MR4567422} to show, for $q \equiv 3 \bmod 4$, that each tentacle has length one. Namely, the reason that there are no tentacles of length at least 2 for $q \equiv 3 \bmod 4$ follows by contradiction: If there were tentacles of length 2 or longer, then some node must have two parents, $(A, B)$ and $(B, A)$, that both have parents, but then both $A^2 - B^2$ and $B^2 - A^2$ would have to be squares, which is impossible (for nontrivial nodes) because $-1$ is not a square \cite{MR4567422}. As differences in backtrackability show that the jellyfish shown come with different sturctures for $q \equiv 3 \bmod 4$ and $q \equiv 5 \bmod 8$, we now turn to the study of backtracking in the AGM.

Analogously to the definition of $n$-times advanceability, for $n \geq 0$, we say a node $(a_0, b_0)$ in $S_K$ is \defn{$n$-times backtrackable} if there exist $(a_{-1}, b_{-1}), (a_{-2}, b_{-2}), \mathellipsis, (a_{-n}, b_{-n}) \in S_K$ such that
\begin{equation}(a_{-n}, b_{-n}) \overset{\text{AGM}}{\mapsto} (a_{1-n}, b_{1-n}) \overset{\text{AGM}}{\mapsto} \cdots \overset{\text{AGM}}{\mapsto} (a_{-1}, b_{-1}) \overset{\text{AGM}}{\mapsto} (a_0, b_0),\end{equation}
and we define $\Sback{n}{K}$ to be the set of these $n$-times backtrackable nodes $(a_0, b_0)$. Similarly, we define $\Sback{\infty}{K} = \bigcap_{n \in \mathbb{N}} \Sback{n}{K}$, denoted the set of indefinitely backtrackable nodes, and we define $F_K^{\mathrm{back}_\infty}$ as the restriction of $F_K$ to $\Sback{\infty}{K}$. Note that $F_K^{\mathrm{back}_\infty}$ still has arrows with directions corresponding to AGM-advancement, not AGM-backtracking. It turns out that similar results for AGM-advancement apply to AGM-backtracking.
\begin{theorem}\label{5mod8backtrackingresults}
    If $q$ is a prime power with $\qfivemodeight$, then
\begin{enumerate}
    \item (Stabilization) $\Sback{\infty}{\Fq}=\Sback{2}{\Fq}.$
    \item $\Sback{\infty}{\Fq}=\set{(a,b)\in S_K: a^2(a^2-b^2)\in (\Fq^\times)^4}$, and is nonempty if and only if $q\neq 5,13.$
    \item $\abs{\Sback{\infty}{\Fq}}=\frac{(q-1)}{4}(q-7-a_q),$
    where $a_q$ is the Frobenius trace of the elliptic curve $y^2=x^3-x$ over $\Fq$.
    \item Each node in $\Sback{\infty}{\Fq}$ has in-degree one in $F_{\Fq}^{\mathrm{back}_\infty}$, corresponding to backtracking over $\Sback{\infty}{\Fq}$ being a single-valued function.
    \item Each connected component of $F_{\Fq}^{\mathrm{back}_\infty}$ has a bell-head with length-two, $Y$-shaped reverse-tentacles (called \defn{colons}) pointing to each of its nodes, constituting a directed analogue of a 2-volcano with depth 2 from \cite{Volcanoes}, an example of which is shown in Figure~\ref{fig:backwards F29 jellyfish}.
    \begin{figure}[h]
      \includegraphics[width=0.6\linewidth]{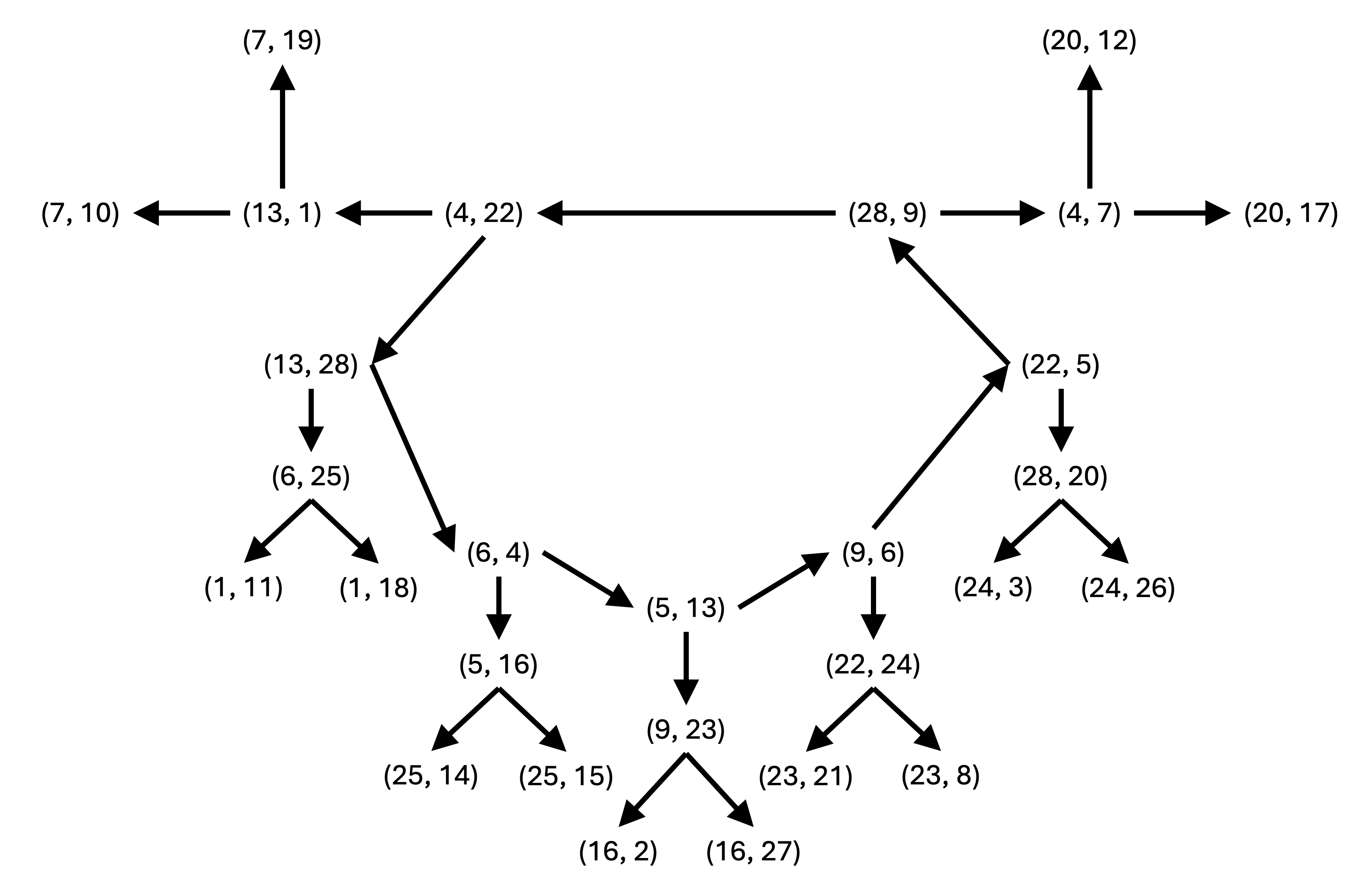}
      \caption{One connected component of $F_{\F_{29}}^{\mathrm{back}_\infty}$, showing the same structure as a connected component of $F_{\F_{29}}^{\mathrm{adv}_\infty}$ except with arrows pointed in the opposite direction}
      \label{fig:backwards F29 jellyfish}
    \end{figure}
\end{enumerate}
\end{theorem}
\begin{remark*}
    While Theorem~\ref{5mod8results} and Corollary~\ref{introduction reversal} imply Theorem~\ref{5mod8backtrackingresults} (5), our proof of Corollary~\ref{introduction reversal} relies on Theorem~\ref{5mod8backtrackingresults} (5). We are not aware of a general construction that gives the isomorphism in Corollary~\ref{introduction reversal}; it is a consequence of our structure theorem.
\end{remark*}

The similarities between Theorem~\ref{5mod8results} and Theorem~\ref{5mod8backtrackingresults} suggests that there is a contravariant isomorphism between $F_{\Fq}^{\mathrm{adv}_\infty}$ and $F_{\Fq}^{\mathrm{back}_\infty}$. This turns out to be the case. In fact, this contravariant isomorphism also applies when $q \equiv 3 \bmod 4$, so the results of Theorem~\ref{3mod4results} can be turned into similar statements about $S_{\Fq}^{\mathrm{back}_\infty}$ and $F_{\Fq}^{\mathrm{back}_\infty}$ in an analogous manner to how the results of Theorem~\ref{5mod8results} can be extended to Theorem~\ref{5mod8backtrackingresults}. (This partially follows from \cite{MR4567422} but also requires results about $F_{\Fq}^{\mathrm{back}_\infty}$ for $\qthreemodfour$ that were not covered in \cite{MR4567422} or in \cite{mcspirit2023hypergeometryagmfinitefields}.)

\begin{corollary}\label{introduction reversal}
    If $q$ is a prime power with $\qthreemodfour$ or $\qfivemodeight$, then $F_{\Fq}^{\mathrm{back}_\infty}$'s reversal (i.e., $F_{\Fq}^{\mathrm{back}_\infty}$ with arrow directions reversed) is isomorphic to $F_{\Fq}^{\mathrm{adv}_\infty}$ as a directed graph. (This shall also be referred to as $F_{\Fq}^{\mathrm{back}_\infty}$ having a contravariant isomorphism with $F_{\Fq}^{\mathrm{adv}_\infty}$.)
\end{corollary}

In view of Corollary~\ref{introduction reversal}, a natural question is whether the conclusion of Corollary~\ref{introduction reversal} holds for arbitrary finite field with odd characteristic. In fact, if one could prove a stronger statement that $F_{\Fq}$'s reversal and $F_{\Fq}$ are isomorphic as directed graphs for any $\Fq$ with odd characteristic, then it would automatically imply the conclusion of Corollary~\ref{introduction reversal} and that the stabilization of $\Sadv{n}{\Fq}$ and $\Sback{n}{\Fq}$ occur at the same $n$.

To this end, we offer Theorem~\ref{population reversal main theorem} as evidence towards the potential contravariant automorphism of $F_{\Fq}$ by demonstrating certain symmetries involving several interesting graphical features of $F_{\Fq}$ that would follow if $F_{\Fq}$ is contravariantly isomorphic to itself. The features involved (the tentacles, colons, and cycles of $F_{\Fq}$) can be detected by $S_{\Fq}^{\mathrm{adv}_\infty}$ and $S_{\Fq}^{\mathrm{back}_\infty}$ as follows:
\begin{align}
    \text{The tentacles in $F_{\Fq}$ are given by } & \Sadv{\infty}{\Fq}\setminus \Sback{\infty}{\Fq}.\notag\\
    \text{The colons in $F_{\Fq}$ are given by } & \Sback{\infty}{\Fq}\setminus \Sadv{\infty}{\Fq}.\notag\\
    \text{The cyclic nodes in $F_{\Fq}$ are given by } & \Sadv{\infty}{\Fq} \cap \Sback{\infty}{\Fq} \text{, which will be referred to as $S_{\Fq}^{\mathrm{cyc}}$.}
\end{align}
Observe that the set of tentacle nodes with depth strictly less than $n$ (where the nodes at the tips of the tentacles have depth 0, their children in the tentacles have depth 1, etc.) is
\begin{equation*}\Sadv{\infty}{\Fq}\setminus \Sback{n}{\Fq}.\end{equation*}
Thus, the maximal tentacle length is given by the smallest nonnegative integer $n$ such that
\begin{equation*}\Sadv{\infty}{\Fq} \setminus \Sback{n}{\Fq} = \Sadv{\infty}{\Fq} \setminus \Sback{\infty}{\Fq}, \quad \text{or equivalently } \Sadv{\infty}{\Fq} \cap \Sback{n}{\Fq} = S_{\Fq}^{\mathrm{cyc}}.\end{equation*}
Similarly, the maximal colon length is the smallest $n$ such that $\Sback{\infty}{\Fq} \cap \Sadv{n}{\Fq} = S_{\Fq}^{\mathrm{cyc}}$.

\begin{theorem}\label{population reversal main theorem}
    If $q$ is any odd prime power, then
    \begin{enumerate}
        \item For $n \geq 0$, $\Sadv{n}{\Fq}=\Sadv{\infty}{\Fq}$ if and only if $\Sback{n}{\Fq}=\Sback{\infty}{\Fq}$.
        \item For $n \geq 0$, $\Sadv{n}{\Fq} \cap \Sback{\infty}{\Fq} = S_{\Fq}^\mathrm{cyc}$ if and only if $\Sback{n}{\Fq} \cap \Sadv{\infty}{\Fq} = S_{\Fq}^\mathrm{cyc}$ (so the maximum colon length equals the maximum tentacle length).
        \item We have $\abs{\Sadv{\infty}{\Fq}}=\abs{\Sback{\infty}{\Fq}}$. Moreover, $\abs{\Sadv{n}{\Fq}}=\abs{\Sback{n}{\Fq}}$ for $n \geq 0$.
    \end{enumerate}
\end{theorem}

We switch focus to note a contrast between the $q \equiv 3 \bmod 4$ and $q \equiv 5 \bmod 8$ cases resulting from the different tentacle structures, as described in part (5) of Theorems~\ref{3mod4results}~and~\ref{5mod8results}. The following population statement already tells the proportion between sizes of different parts of the jellyfish structure for $\qfivemodeight$. This is analogous to how $\abs{S_{\Fq}^{\mathrm{cyc}}}~=~\frac{1}{2}\abs{\Sadv{\infty}{\Fq}}$ when $\qthreemodfour$, but the longer, branched tentacles lead to a larger proportion of noncyclic nodes.

\begin{proposition}\label{cycle fractions}
    If $\qfivemodeight$, then
    $\abs{S_{\Fq}^{\mathrm{cyc}}}=\frac{1}{4}\abs{\Sadv{\infty}{\Fq}}$.
\end{proposition}

We have not as of yet obtained the analogue for Theorems \ref{3mod4results}~and~\ref{5mod8results} for finite fields $\Fq$ with $q \equiv 1 \bmod 8$. It would be very interesting to get the analogous results for $q \equiv 1 \bmod 8$, but computing these for $q \equiv 1 \bmod 8$ up to some range, we have discovered examples where the maximum tentacle length is 3, 4, $\dots$, or 9, with no indication of any upper bound,\footnote{In an upcoming work, we will prove that arbitrarily high tentacle lengths do indeed occur.} so we at present are not able to speculate on the true nature of the AGM in these cases.

In Section~\ref{Determining advancement criteria}, we discuss advanceability criteria of the AGM over $\Fq$ with $q \equiv 1 \bmod 4$, particularly focusing on when $q \equiv 5 \bmod 8$. In Section~\ref{Determining indefinitely advanceable node counts}, we use a simpler version of the AGM graph, which only keeps track of the ratio between the entries of each node, to determine the population of $\Sadv{\infty}{\Fq}$ for $q \equiv 5 \bmod 8$. In Section~\ref{Backtracking}, we focus on backtracking criteria of the AGM over finite fields. In Sections~\ref{Reversal symmetry}~and~\ref{Jellyfish structure}, we discuss the structure of jellyfish, with the former section focusing on contravariant isomorphism. In Section~\ref{Proof of theorems}, we relate the results developed along the way to the results stated in the introduction.

This paper shall use the word jellyfish to refer to a connected component of the directed graph over $S_K^{\mathrm{adv}_\infty}$ corresponding to AGM-advancement for any finite field $K$ where $-1$ is not a fourth power. This is a slight generalization of its definition in \cite{MR4567422} and \cite{mcspirit2023hypergeometryagmfinitefields}, where it only applied to fields where $-1$ is not a square.

\begin{remark*}
    In the spirit of making the premises as general as possible without complicating the proof, the following results are given with premises that are not restricted to finite fields if possible. For example, if the only important part is that $-1$ is a square but not a fourth power, instead of saying, ``Let $q$ be a prime power with $q \equiv 5 \bmod 8$," the premise will instead be, ``Let $K$ be a field where $-1$ is a square but not a fourth power."
\end{remark*}

\section{Determining advancement criteria}\label{Determining advancement criteria}
In this section, we will determine which nodes can be advanced indefinitely in $\mathbb{F}_q$ for $q \equiv 5 \bmod 8$, thus proving several parts of Theorem~\ref{5mod8results}. As a first part step, the following lemma will be useful, as it says we don't need to separately worry about whether each individual pair $(\alpha,\beta)\in K^2$ in an AGM-advancement chain is in $S_K$.

\begin{proposition}\label{nontrivial inheritance}
    Let $K$ be a field with characteristic not equal to two, and let $(\alpha, \beta) \overset{\AGM_K}{\mapsto} (\gamma, \delta)$. Then $(\alpha, \beta)$ is nontrivial if and only if $(\gamma, \delta)$ is nontrivial.
\end{proposition}

\begin{proof}[Proof of Proposition~\ref{nontrivial inheritance}]
    It follows from the assumption that
    $$\alpha\beta(\alpha+\beta)(\alpha-\beta)^2=8\gamma\delta^2(\gamma+\delta)(\gamma-\delta),$$
    and hence, $\alpha,\beta,\alpha\pm \beta \neq 0$ if and only if $\gamma,\delta,\gamma\pm \delta \neq 0$.
\end{proof}

The preceding proposition and the following two lemmas will help establish parts of Theorem~\ref{5mod8results}.

\begin{lemma}\label{fourth power lemma}
    Let $K$ be a field with characteristic not equal to two. Any $(a_0, b_0) \in S_K$ can be advanced at least twice if and only if $\left(\frac{a_0+b_0}{2}\right)^2a_0b_0$ is a fourth power. In this case, if $-1$ is a square in $K$, then both of $(a_0, b_0)$'s children can be advanced at least once.
\end{lemma}

\begin{proof}[Proof of Lemma~\ref{fourth power lemma}]
    Assume $(a_0,b_0)\mapsto (a_1,b_1)\mapsto (a_2,b_2)$. Then
    $$\left(\frac{a_0+b_0}{2}\right)^2a_0b_0={a_1}^2{b_1}^2={b_2}^4,$$
    which is clearly a fourth power in $K$.

    Conversely, assume $4a_0b_0(a_0+b_0)^2$ is a fourth power in $K$. Then $a_0b_0$ is a square, so $(a_0,b_0)$ has two children, $(a_1,b_1)$ and $(a_1,-b_1)$, where $2a_1=a_0+b_0$ and ${b_1}^2=a_0b_0$. Note that $4a_0b_0(a_0+b_0)^2=16{a_1}^2{b_1}^2$, so the two square roots of $4a_0b_0(a_0+b_0)^2$ are $\pm 4a_1b_1$. Since $4a_0b_0(a_0+b_0)^2$ is a fourth power, at least one of its square roots is a square. However, since $-1$ is a square in $K$, $4a_1b_1$ is a square if and only if $-4a_1b_1$ is a square. It follows that both $\pm 4a_1b_1$ are squares, which imply that both $(a_1,\pm b_1)$ are advanceable.
\end{proof}

\begin{lemma}\label{advanceability extension}
    Let $q$ be a prime power with $q \equiv 5 \bmod 8$. If $(a_0, b_0) \in \Sadv{2}{\Fq}$, then exactly one of its children is in $\Sadv{2}{\Fq}$.
\end{lemma}

\begin{proof}[Proof of Lemma~\ref{advanceability extension}]
    Let $(a_0, b_0)$'s children be $\left(a_1, b_1\right)$ and $\left(a_1, -b_1\right)$, so $a_1 = \frac{a_0+b_0}{2}$ and ${b_1}^2 = a_0b_0$. Consider
    $$A = \left(\frac{a_1+b_1}{2}\right)^2\!a_1b_1 \qquad \text{ and } \qquad B = \left(\frac{a_1-b_1}{2}\right)^2\!a_1(-b_1).$$
    Since $\left(a_1, b_1\right)$ and $\left(a_1, -b_1\right)$ are both in $\Sadv{1}{\Fq}$ by Lemma~\ref{fourth power lemma}, $A$ and $B$ are both squares. However, their product
    $$AB = -\left(\frac{a_0-b_0}{4}\right)^4 {b_0}^4$$
    is not a fourth power (since $(a_0, b_0)$ is nontrivial by Proposition~\ref{nontrivial inheritance} and $-1$ is not a fourth power), so exactly one of $A$ and $B$ is a fourth power, corresponding to exactly one of $\left(a_1, b_1\right)$ and $\left(a_1, -b_1\right)$ being in $\Sadv{2}{\Fq}$ by Lemma~\ref{fourth power lemma}.
\end{proof}

\begin{remark*}
    While the product of two squares not being a fourth power does not necessitate that one of the squares is a fourth power in infinite fields, it does necessitate that they are not both fourth powers, so infinite fields where $-1$ is a square but not a fourth power obey the property that no nontrivial node has two twice-advanceable children, so AGM-advancement is unique on $\Sadv{\infty}{K}$.
\end{remark*}

Now that we know that (for $\qfivemodeight$,) each member of $\Sadv{2}{\Fq}$ has a child in $\Sadv{2}{\Fq}$, this allows us to conclude that each member of $\Sadv{2}{\Fq}$ is indefinitely advanceable.

\begin{corollary}\label{SFqadm formula}
    Let $q\equiv 5\bmod{8}$. Then the set of all indefinitely advanceable nontrivial pairs is given by $$S_{\F_q}^{\mathrm{adv}_\infty} = S_{\F_q}^{\mathrm{adv}_2} = \left\{(a,b) \in S_{\mathbb{F}_q} \mid \left(\frac{a+b}{2}\right)^2ab \text{ is a fourth power}\right\}\!.$$
    Moreover, $F_{\F_q}^{\mathrm{adv}_\infty}$ is a (single-valued) function.
\end{corollary}

\begin{proof}[Proof of Corollary~\ref{SFqadm formula}]
    Lemma~\ref{fourth power lemma} tells us that $S_{\F_q}^{\mathrm{adv}_2} = \left\{(a,b) \in S_{\mathbb{F}_q} \mid \left(\frac{a+b}{2}\right)^2ab \text{ is a fourth power}\right\}$. Repeatedly applying Lemma~\ref{advanceability extension} to members of $S_{\F_q}^{\mathrm{adv}_2}$ and their twice-advanceable descendants tells us that $S_{\F_q}^{\mathrm{adv}_2} \subseteq S_{\F_q}^{\mathrm{adv}_\infty}$, and $S_{\F_q}^{\mathrm{adv}_\infty} \subseteq S_{\F_q}^{\mathrm{adv}_2}$ by the definition of $S_{\F_q}^{\mathrm{adv}_\infty}$, so $S_{\F_q}^{\mathrm{adv}_\infty} = S_{\F_q}^{\mathrm{adv}_2}$ follows.

    For the proof that $F_{\F_q}^{\mathrm{adv}_\infty}$ is a single-valued function, Lemma~\ref{advanceability extension} tells us that each node in $S_{\F_q}^{\mathrm{adv}_\infty}$ has only one child in $S_{\F_q}^{\mathrm{adv}_2}$ and therefore cannot have multiple children in $S_{\F_q}^{\mathrm{adv}_\infty}$.
\end{proof}

\begin{remark*}
    This also applies to $q \equiv 3 \bmod 4$ (because all squares are fourth powers in finite fields where $-1$ is not a square), allowing this result for $q \equiv 5 \bmod 8$ to be viewed as a generalization of the $q \equiv 3 \bmod 4$ case.
\end{remark*}

\section{Determining indefinitely advanceable node counts}\label{Determining indefinitely advanceable node counts}
For some purposes, we will be interested in the ratios $\frac{b}{a}$; this ratio shall be referred to as $k$ in order to be consistent with \cite{Borwein_book} and a future paper about the Galois-theoretic properties of the polynomials that references results from that source. This $k$ is related to $\lambda$ in \cite{Borwein_book}, \cite{MR4567422}, \cite{kayath2024agmaquariumsellipticcurves}, and \cite{mcspirit2023hypergeometryagmfinitefields}; specifically, $k^2 = \lambda$.

This ratio is important because the effect that AGM advancement of $(a, b)$ has on $k$ does not depend on the precise values of $a$ and $b$, only on their ratio $k = \frac{b}{a}$. Specifically, if $(a_1, b_1) \overset{\text{AGM}}{\mapsto} (a_2, b_2)$, we have that $k_2 = \frac{b_2}{a_2}$ is related to $k_1 = \frac{b_1}{a_1}$ by $(1+k_1)^2 {k_2}^2 = 4 k_1$ \cite[\S 4]{Borwein_book}; when $k_1$ and $k_2$ satisfy this equation, we say that $k_1$ advances to $k_2$ under the rules of $k$-advancement, notated $k_1 \overset{\mathrm{k}}{\mapsto} k_2$, or that $k_2$ backtracks to $k_1$ under the rules of $k$-advancement.

Let $T_K = K \setminus \{0, 1, -1\}$; these $k$-values correspond to nontrivial nodes and therefore shall be referred to as nontrivial $k$-values. Define $G_K$ to be the subset of ${T_K}^2$ given by $G_K = \{(k_1, k_2) \in {T_K}^2 \mid (1+k_1)^2 {k_2}^2 = 4 k_1\}$. Let $T_K^{\mathrm{adv}_0} = T_K^{\mathrm{back}_0} = T_K$, $T_K^{\mathrm{adv}_n} = \{k_1: (k_1, k_2) \in G_K \cap (T_K^{\mathrm{adv}_{n-1}} \times T_k)\}$ and $T_K^{\mathrm{back}_n} = \{k_2: (k_1, k_2) \in G_K \cap (T_k \times T_K^{\mathrm{back}_{n-1}})\}$ for $n \in \mathbb{N}$, $T_K^{\mathrm{adv}_\infty} = \bigcup_{n \in \mathbb{N}} T_K^{\mathrm{adv}_n}$, and $T_K^{\mathrm{back}_\infty} = \bigcup_{n \in \mathbb{N}} T_K^{\mathrm{back}_n}$. All of these definitions are analogous to for $S_K$ and $F_K$, and, in fact, we will see in Lemma~\ref{arrow biconditional} and Corollary~\ref{improved biconditional} that $(a, b) \mapsto \frac{b}{a}$ commutes with advancement, so $T_K^{\mathrm{adv}_n}~=~\left\{\frac{b}{a}: (a, b) \in S_K^{\mathrm{adv}_n}\right\}$ and $T_K^{\mathrm{back}_n} = \left\{\frac{b}{a}: (a, b) \in S_K^{\mathrm{back}_n}\right\}$.

\begin{lemma}\label{arrow biconditional}
    Let $(a_0,b_0), (a_1,b_1) \in S_K$. Then $(a_0,b_0)\overset{\AGM}{\mapsto} (a_1,b_1)$ if and only if $2a_1=a_0+b_0$ and $\frac{b_0}{a_0} \overset{\mathrm{k}}{\mapsto} \frac{b_1}{a_1}$.
\end{lemma}

\begin{proof}[Proof of Lemma~\ref{arrow biconditional}]
    If $2a_1 = a_0 + b_0$, then we have ${b_1}^2 - a_0b_0 = 4{a_0}^2 \times \left(\left(1 + \frac{b_0}{a_0}\right)^2 \left(\frac{b_1}{a_1}\right)^2 - 4\frac{b_0}{a_0}\right)$, so ${b_1}^2 = a_0b_0$ if and only if $\left(1 + \frac{b_0}{a_0}\right)^2 \left(\frac{b_1}{a_1}\right)^2 = 4\frac{b_0}{a_0}$, i.e., if and only if $\frac{b_0}{a_0} \overset{\mathrm{k}}\mapsto \frac{b_1}{a_1}$.
\end{proof}

Lemma~\ref{arrow biconditional} allows us to relate AGM advancement to $k$-advancement, which in turn allows us to relate the populations of $S_K^{\mathrm{adv}_n}$ and $T_K^{\mathrm{adv}_n}$ and to relate the populations of $S_K^{\mathrm{back}_n}$ and $T_K^{\mathrm{back}_n}$. This will be useful for proving part (3) of Thoerems~\ref{5mod8results}~and~\ref{5mod8backtrackingresults}.

\begin{corollary}\label{improved biconditional}
    For any $n \in \mathbb{N} \cup \{\infty\}$ and any field $K$ with characteristic not equal to two, we have $S_K^{\mathrm{adv}_n} = \{(a, ka): a \in K^\times, k \in T_K^{\mathrm{adv}_n}\}$ and $S_K^{\mathrm{back}_n} = \{(a, ka): a \in K^\times, k \in T_K^{\mathrm{back}_n}\}$.

    In particular, $$\left| S_{\F_q}^{\mathrm{adv}_\infty} \right| = \left|K^\times\right| \times \left| T_{\F_q}^{\mathrm{adv}_\infty} \right|\!.$$
\end{corollary}

\begin{proof}[Proof of Corollary~\ref{improved biconditional}]
    If we have an AGM-advancement chain $(a_0, b_0) \overset{\AGM}{\mapsto} (a_1, b_1) \overset{\AGM}{\mapsto} \cdots \overset{\AGM}{\mapsto} (a_n, b_n)$, Lemma~\ref{arrow biconditional} lets us take the $k$-value of each node to get a $k$-advancement chain $k_0 \overset{\mathrm{k}}{\mapsto} k_1 \overset{\mathrm{k}}{\mapsto} \cdots \overset{\mathrm{k}}{\mapsto} k_n$, so $S_K^{\mathrm{adv}_n} \subseteq \{(a, ka): a \in K^\times, k \in T_K^{\mathrm{adv}_n}\}$, and $S_K^{\mathrm{back}_n} \subseteq \{(a, ka): a \in K^\times, k \in T_K^{\mathrm{back}_n}\}$.

    If we instead start with a $k$-advancement chain $k_0 \overset{\mathrm{k}}{\mapsto} k_1 \overset{\mathrm{k}}{\mapsto} \cdots \overset{\mathrm{k}}{\mapsto} k_n$ and some starting $a_0 \in K^\times$, we can use Lemma~\ref{arrow biconditional} to obtain an AGM-advancement chain $(a_0, k_0 a_0) \overset{\AGM}{\mapsto} (a_1, k_1 a_1) \overset{\AGM}{\mapsto} \cdots \overset{\AGM}{\mapsto} (a_n, k_n a_n)$ by recursively setting $a_{i+1} = \frac{a_i + k_i a_i}{2}$ for $i = 0, 1, 2, \mathellipsis, n-1$, so $S_K^{\mathrm{adv}_n} \supseteq \{(a, ka): a \in K^\times, k \in T_K^{\mathrm{adv}_n}\}$. Analogously, if we start with a $k$-advancement chain $k_0 \overset{\mathrm{k}}{\mapsto} k_1 \overset{\mathrm{k}}{\mapsto} \cdots \overset{\mathrm{k}}{\mapsto} k_n$ and some starting $a_n \in K^\times$, we can use Lemma~\ref{arrow biconditional} to obtain an AGM-advancement chain $(a_0, k_0 a_0) \overset{\AGM}{\mapsto} (a_1, k_1 a_1) \overset{\AGM}{\mapsto} \cdots \overset{\AGM}{\mapsto} (a_n, k_n a_n)$ by recursively setting $a_{i-1} = \frac{2 a_i}{1 + k_i}$ for $i = n, n-1, n-2, \mathellipsis, 1$, so $S_K^{\mathrm{back}_n} \supseteq \{(a, ka): a \in K^\times, k \in T_K^{\mathrm{back}_n}\}$.
\end{proof}

Corollary~\ref{improved biconditional} allows us to translate what we already knew about node advanceability, such as Lemma~\ref{SFqadm formula}, into information about $k$-advanceability.

\begin{corollary}\label{T count}
    If $q \equiv 5 \bmod{8}$, then
    $$T_{\F_q}^{\mathrm{adv}_\infty} = \{k \in T_{\F_q} \mid 4k\left(1+k\right)^2 \text{ is a fourth power}\}$$
\end{corollary}

\begin{proof}[Proof of Corollary~\ref{T count}]
    Corollary~\ref{improved biconditional} tells us that $k \in T_{\F_q}^{\mathrm{adv}_\infty}$ if and only if $\left(2, 2k\right) \in S_{\F_q}^{\mathrm{adv}_\infty}$. The conclusion follows from substituting $(a, b) = (2, 2k)$ into Lemma~\ref{SFqadm formula}.
\end{proof}

Now that we have a formula for $T_{\F_q}^{\mathrm{adv}_\infty}$, we can use it to figure out its population. Note that $T_{\F_q}^{\mathrm{adv}_\infty}$ consists of the $k$-coordinates superelliptic curve $y^4 = 4k\left(1+k\right)^2$, which has genus one, so its population is related to the population of some elliptic curve. Specifically, it turns out to be related to $y^2=2x(1+x^2)$.

\begin{theorem}\label{Legendre elliptic curve counts}
    Let $q\equiv 5\bmod{8}$. Then the number of all indefinitely advanceable $k$-values is given by $$\left| T_{\F_q}^{\mathrm{adv}_\infty} \right| = \frac{q-a_q-7}{4},$$
    where $a_q$ denotes the trace of Frobenius of the elliptic curve $y^2=x^3-x$ over $\mathbb{F}_q$.
\end{theorem}

\begin{proof}[Proof of Theorem~\ref{Legendre elliptic curve counts}]
    Recall from Corollary~\ref{T count} that $\left| T_{\F_q}^{\mathrm{adv}_\infty} \right| = \left\{k \in T_{\F_q} \mid 4k\left(1+k\right)^2 \text{ is a fourth power}\!\right\}$. Since $1+k$ cannot be 0, in order for $4k\left(1+k\right)^2$ to be a square, $k$ must be a square. Because $k = 0$ is already excluded, each applicable value of $k$ has two square roots, so we have:
    \begin{align*}
        \left| T_{\F_q}^{\mathrm{adv}_\infty} \right| &= \#\!\left\{k \in T_{\F_q} \mid 4k(1+k)^2 \text{ is a fourth power}\right\}
        \\&= \#\!\left\{k \in \mathbb{F}_q \setminus \{-1, 0, 1\} \mid 4k(1+k)^2 \text{ is a fourth power}\right\}
        \\&= \frac{\#\!\left\{x \in \mathbb{F}_q \mid x^2 \notin \{-1, 0, 1\}, 2x\!\left(1+x^2\right) \text{ is a square}\right\}}{2} & & k = x^2
        \\&= \frac{\#\!\left\{(x, y) \in {\mathbb{F}_q}^2 \mid x^2 \notin \{-1, 0, 1\}, y^2=2x\!\left(1+x^2\right)\right\}}{4} & & x^2 \notin \{0, -1\} \Rightarrow 2x\!\left(1+x^2\right) \neq 0
        \\&= \frac{\#E_{y^2=2x\left(x^2+1\right)}\!\left(\mathbb{F}_q\right)-8}{4}
        \\&= \frac{q-a_q-7}{4},
    \end{align*}
    where the $-8$ is to exclude $(0, 0)$, $(1, 0)$, $(-1, 0)$, $(i, 1-i)$, $(i, -1+i)$, $(-i, 1+i)$, $(-i, -1-i)$, and the projective point at infinity, and $a_q = q + 1 - \#E_{y^2=2x(1+x^2)}\!\left(\mathbb{F}_q\right)$, where $\#E_{y^2=2x(1+x^2)}\!\left(\mathbb{F}_q\right)$ includes the projective point at infinity.
    
    The elliptic curve $y^2=2x(1+x^2)$ is isomorphic over $\mathbb{Q}$ to the elliptic curve $y^2=x^3+4x$. For $q \equiv 1 \bmod 4$, it is also isomorphic over $\mathbb{F}_q$ to $y^2=x^3-x$ because $-4$ is a fourth power in $\mathbb{F}_q$.
\end{proof}

Because the populations of $S_{\F_q}^{\mathrm{adv}_\infty}$ and $T_{\F_q}^{\mathrm{adv}_\infty}$ are related by Corollary~\ref{improved biconditional}, knowledge of the latter allows us to deduce the former.

\begin{corollary}\label{SFqadvinfty population}
    Let $q\equiv 5\bmod{8}$. Then
    $$\left| S_{\F_q}^{\mathrm{adv}_\infty} \right| = \frac{q-1}{4} \times \left(q-a_q-7\right),$$
    where $a_q$ again denotes the trace of Frobenius of the elliptic curve $y^2=x^3-x$ over $\Fq$.
\end{corollary}

\begin{proof}[Proof of Corollary~\ref{SFqadvinfty population}]
    This follows from substituting Theorem~\ref{Legendre elliptic curve counts} into Corollary~\ref{improved biconditional}.
\end{proof}

This result guarantees that all prime powers $q \equiv 5 \bmod 8$ past a certain point have nonempty jellyfish swarms.

\begin{corollary}\label{nonempty past 13}
    For $q\equiv 5\bmod{8}$, $S_{\F_q}^{\mathrm{adv}_\infty}$ and $T_{\F_q}^{\mathrm{adv}_\infty}$ are empty if and only if $q \in \{5, 13\}$.
\end{corollary}

\begin{proof}[Proof of Corollary~\ref{nonempty past 13}]
    It can be directly checked from \eqref{trace of frobenius formula} that $q-a_q-7=0$ for $q=5,13$. For $q > 9+4\sqrt{2}$ (about 14.7), Hasse's theorem guarantees that $q-a_q-7\geq q-2\sqrt{q}-7>0$.
\end{proof}

\section{Backtracking}\label{Backtracking}
Now that we have investigated the structure of $\Sadv{\infty}{\Fq}$, it makes sense to investigate the structure of $\Sback{\infty}{\Fq}$. In addition to proving Theorem~\ref{5mod8backtrackingresults} as an analogue of Theorem~\ref{5mod8results}, this combined with the previous section will help motivate the contravariant isomorphism between $T_{\Fq}^{\mathrm{adv}_\infty}$ and $T_{\Fq}^{\mathrm{back}_\infty}$, yielding reversal symmetry, which is a powerful tool that will help us, among other things, prove Theorem~\ref{population reversal main theorem}. First, we develop analogs of previous advancement results for backtracking. Specifically, Lemma~\ref{fourth power lemma} corresponds to Corollary~\ref{parental backtrackability} and Lemma~\ref{backtracking twice}, Lemma~\ref{advanceability extension} corresponds to Lemma~\ref{backtrackability extension}, and Corollary \ref{SFqadm formula} corresponds to Remark~\ref{unique backtracking} and Corollary~\ref{backtracking twice to infinite}.

\begin{lemma}[Generalization of Corollary 2.2 in \cite{kayath2024agmaquariumsellipticcurves}]\label{backtracking once}
    Let $K$ be a field with characteristic not equal to two. A node $\left(a_n, b_n\right) \in S_K$ can be backtracked at least once if and only if ${a_n}^2-{b_n}^2$ is a square in $K$. In this case, it has exactly two parents, which are each others' reversals.
\end{lemma}

\begin{proof}[Proof of Lemma~\ref{backtracking once}]
    The coordinates of any parent of $\left(a_n, b_n\right)$ must have sum $2a_n$ and product ${b_n}^2$, so they must be the roots of $x^2-2a_n x+{b_n}^2=0$ in either order. The result (including that the two roots are distinct for nontrivial nodes) follows from using the quadratic equation.
\end{proof}

\begin{remark*}
    When $K = \mathbb{R}$ and $a_n$ and $b_n$ are both positive, this becomes the AM–GM inequality.
\end{remark*}

\begin{remark*}
    If $\left(a_n, b_n\right)$ is backtrackable at least once and $\left(a_{n-1},b_{n-1}\right) $ is either of the two parents, then $4{a_n}^2-4{b_n}^2=(a_{n-1}-b_{n-1})^2$.
\end{remark*}

\begin{corollary}\label{parental backtrackability}
    Let $K$ be a field with characteristic not equal to two and where $-1$ is a square. If $\left(a_n, b_n\right) \in S_K^{\mathrm{back}_1}$, then its parents are either both in $S_K^{\mathrm{back}_1}$ or neither in $S_K^{\mathrm{back}_1}$.
\end{corollary}

\begin{proof}[Proof of Corollary~\ref{parental backtrackability}]
    Let $\left(a_{n-1}, b_{n-1}\right)$ and $\left(b_{n-1}, a_{n-1}\right)$ be the parents of $\left(a_n, b_n\right)$. By Lemma~\ref{backtracking once}, that these are either both or neither backtrackable is equivalent to saying that ${a_{n-1}}^2-{b_{n-1}}^2$ is a square if and only if ${b_{n-1}}^2-{a_{n-1}}^2$ is a square (where $\left(a_{n-1}, b_{n-1}\right)$ and $\left(b_{n-1}, a_{n-1}\right)$ are the parents of $(a_n, b_n)$), which is true because $-1$ is a square.
\end{proof}

\begin{lemma}\label{backtracking twice}
    Let $K$ be a field with characteristic not equal to two and where $-1$ is a square. Any $(a_n, b_n) \in \Sback{1}{K}$ can be backtracked at least twice if and only if ${a_n}^2\!\left({a_n}^2-{b_n}^2\right)$ is a fourth power.
\end{lemma}

\begin{proof}[Proof of Lemma~\ref{backtracking twice}]
    Assume $(a_{n-2},b_{n-2})\mapsto (a_{n-1},b_{n-1})\mapsto (a_n,b_n)$. Then
    $${a_n}^2\!\left({a_n}^2-{b_n}^2\right) = \left(\frac{a_{n-2}-b_{n-2}}{4}\right)^4,$$
    which is clearly a fourth power in $K$.

    Conversely, assume ${a_n}^2\!\left({a_n}^2-{b_n}^2\right)$ is a fourth power in $K$. Then ${a_n}^2-{b_n}^2$ is a square, so $(a_n,b_n)$ has two parents, $(a_{n-1},b_{n-1})$ and $(b_{n-1},a_{n-1})$, where $2a_n=a_{n-1}+b_{n-1}$ and ${b_n}^2=a_{n-1}b_{n-1}$. Note that ${a_n}^2\!\left({a_n}^2-{b_n}^2\right)=\frac{({a_{n-1}}^2-{b_{n-1}}^2)^2}{16}$, so the two square roots of ${a_n}^2-{b_n}^2$ are $\pm \frac{{a_{n-1}}^2-{b_{n-1}}^2}{4}$. Since ${a_n}^2\!\left({a_n}^2-{b_n}^2\right)$ is a fourth power, at least one of its square roots is a square. However, since $-1$ is a square in $K$, $\frac{{a_{n-1}}^2-{b_{n-1}}^2}{4}$ is a square if and only if $\frac{{b_{n-1}}^2-{a_{n-1}}^2}{4}$ is a square. It follows that both $\pm \frac{{a_{n-1}}^2-{b_{n-1}}^2}{4}$ are squares, which imply that both $(a_{n-1}, b_{n-1})$ and $(b_{n-1}, a_{n-1})$ are backtrackable by Lemma~\ref{backtracking once}.
\end{proof}

\begin{lemma}\label{backtrackability extension}
    Let $q$ be a prime power with $q \equiv 5 \bmod 8$. If $(a_n, b_n) \in \Sback{2}{\Fq}$, then exactly one of its children can be backtracked at least twice (so $(a_n, b_n)$ can be backtracked at least three times).
\end{lemma}

\begin{proof}[Proof of Lemma~\ref{backtrackability extension}]
    Assume $(a_{n-2},b_{n-2})\mapsto (a_{n-1},b_{n-1})\mapsto (a_n,b_n)$, so $(b_{n-1},a_{n-1})\mapsto (a_n,b_n)$. Consider
    $${a_{n-1}}^2\!\left({a_{n-1}}^2-{b_{n-1}}^2\right) \qquad \text{ and } \qquad {b_{n-1}}^2\!\left({b_{n-1}}^2-{a_{n-1}}^2\right)\!,$$
    which are squares by Corollary~\ref{parental backtrackability} because $(a_{n-1},b_{n-1})$ and $(b_{n-1},a_{n-1})$ are both in $S_{\Fq}^{\mathrm{back}_1}$. However, their product
    $${a_{n-1}}^2\!\left({a_{n-1}}^2-{b_{n-1}}^2\right) \times {b_{n-1}}^2\!\left({b_{n-1}}^2-{a_{n-1}}^2\right) = -{b_n}^4\left(\frac{a_{n-2}-b_{n-2}}{2}\right)^4,$$
    is not a fourth power (since $(a_{n-2}, b_{n-2})$ is nontrivial by Proposition~\ref{nontrivial inheritance} and $-1$ is not a fourth power), so exactly one of ${a_{n-1}}^2\!\left({a_{n-1}}^2-{b_{n-1}}^2\right)$ and ${b_{n-1}}^2\!\left({b_{n-1}}^2-{a_{n-1}}^2\right)$ is a fourth power, corresponding to exactly one of $\left(a_{n-1}, b_{n-1}\right)$ and $\left(b_{n-1}, a_{n-1}\right)$ being at least twice backtrackable by Lemma~\ref{backtracking twice}.
\end{proof}

\begin{remark}\label{unique backtracking}
    While the product of two squares not being a fourth power does not necessitate that one of the squares is a fourth power in infinite fields, it does necessitate that they are not both fourth powers, so infinite fields where $-1$ is a square but not a fourth power obey the property that no node in $S_K^{\mathrm{back}_2}$ has multiple parents in $S_K^{\mathrm{back}_2}$.
\end{remark}

\begin{corollary}\label{backtracking twice to infinite}
    Let $q\equiv 5\bmod{8}$. Then $$S_{\Fq}^{\mathrm{back}_\infty} = S_{\Fq}^{\mathrm{back}_2} = \!\left\{(a, b) \in S_{\Fq} \mid a^2(a^2-b^2) \text{ is a fourth power}\right\}.$$
\end{corollary}

\begin{proof}[Proof of Corollary~\ref{backtracking twice to infinite}]
    The second equality follows from Lemmas \ref{backtracking once}~and~\ref{backtracking twice}. For the first equality, that $S_{\Fq}^{\mathrm{back}_\infty} \subseteq S_{\Fq}^{\mathrm{back}_2}$ is true by definition, and the proof that $S_{\Fq}^{\mathrm{back}_\infty} \supseteq S_{\Fq}^{\mathrm{back}_2}$ is analogous to that of Corollary~\ref{SFqadm formula}, following from repeatedly applying Lemma~\ref{backtrackability extension} to any node in $S_{\Fq}^{\mathrm{back}_2}$ and its twice-backtrackable ancestors.
\end{proof}

\begin{remark}\label{lambda backtracking criteria}
    Recall from Corollary~\ref{improved biconditional} that
    $$k \in T_K^{\mathrm{back}_n} \Longleftrightarrow (a, k a) \in S_K^{\mathrm{back}_n}.$$
    Thus, in a finite field with order congruent to five modulo eight, a nontrivial (i.e., not equal to 0, 1, $-1$, or $\infty$) $k$-value is indefinitely backtrackable if and only if $1-k^2$ is a fourth power, corresponding to the curve $\mu^4=1-k^2$. If we recall its advanceability counterpart, the twice-advanceability elliptic curve $y^2=2x\!\left(1+x^2\right)$ from the proof of Theorem~\ref{Legendre elliptic curve counts}, the two turn out to be related by the birational map $(k, \mu) \mapsto \!\left(\frac{1-x^2}{1+x^2}, \frac{y}{1+x^2}\right)$ with inverse $(x, y) \mapsto \!\left(\frac{\mu^2}{k+1}, \frac{2\mu}{k+1}\right)$. The remark at the beginning of section \ref{Reversal symmetry} showed that this is not an accident.
\end{remark}

The following theorem is important in proving part (5) of Theorem~\ref{5mod8results}, part (5) of Theorem~\ref{5mod8backtrackingresults}, and part (2) of Corollary~\ref{cycle fractions}.

\begin{theorem}\label{quarter overlap}
    Let $q$ be a prime power with $q \equiv 5 \bmod 8$. Exactly one quarter of the members of $S_{\F_q}^{\mathrm{adv}_\infty}$ are in $S_{\F_q}^{\mathrm{back}_\infty}$, and exactly one quarter of the members of $S_{\F_q}^{\mathrm{back}_\infty}$ are in $S_{\F_q}^{\mathrm{adv}_\infty}$.
\end{theorem}

\begin{proof}[Proof of Theorem~\ref{quarter overlap}]
    Let $(a, b)$ be a member of $S_{\F_q}^{\mathrm{adv}_\infty}$. Then Lemma~\ref{fourth power lemma} and Corollary~\ref{SFqadm formula} tell us that $(a, b)$ has exactly one grandchild in $S_{\F_q}^{\mathrm{adv}_\infty}$. Proposition~\ref{nontrivial inheritance}, Lemma~\ref{backtracking once}, and Corollary~\ref{parental backtrackability} tell us that this grandchild has exactly four grandparents, which are all in $S_{\F_q}^{\mathrm{adv}_\infty}$, and Lemma~\ref{backtracking once}, Theorem~\ref{backtrackability extension} and Corollary~\ref{backtracking twice to infinite} tell us that exactly one of these grandparents is in $S_{\F_q}^{\mathrm{back}_\infty}$. Thus, $S_{\F_q}^{\mathrm{adv}_\infty}$ can be partitioned into disjoint subsets each consisting of four nodes (sharing a common indefinitely advanceable grandchild) such that exactly one member of each four-node subset is in $S_{\F_q}^{\mathrm{back}_\infty}$, so exactly one quarter of $S_{\F_q}^{\mathrm{adv}_\infty}$ is in $S_{\F_q}^{\mathrm{back}_\infty}$.

    Analogously, let $(a, b)$ be a member of $S_{\F_q}^{\mathrm{back}_\infty}$. Then Proposition~\ref{nontrivial inheritance}, Lemma~\ref{backtracking once}, and Corollary~\ref{parental backtrackability} tell us that $\left(a, b\right)$ has exactly one grandparent in $S_{\F_q}^{\mathrm{adv}_\infty}$. Lemma~\ref{fourth power lemma} tells us that this grandparent has exactly four grandchildren, which are all in $S_{\F_q}^{\mathrm{back}_\infty}$. Furthermore, Lemma~\ref{advanceability extension} and Corollary~\ref{SFqadm formula} tell us that exactly one of these grandchildren is in $S_{\F_q}^{\mathrm{adv}_\infty}$. Thus, $S_{\F_q}^{\mathrm{back}_\infty}$ can be partitioned into disjoint subsets each consisting of four nodes (sharing a common indefinitely backtrackable grandparent) such that exactly one member of each four-node subset is in $S_{\F_q}^{\mathrm{adv}_\infty}$, so exactly one quarter of $S_{\F_q}^{\mathrm{back}_\infty}$ is in $S_{\F_q}^{\mathrm{adv}_\infty}$.
\end{proof}

Note that we are now able to prove a special case of part (2) of Theorem~\ref{population reversal main theorem}, as the following corollary makes clear.

\begin{corollary}\label{advance backtrack count equality}
    Let $q$ be a prime power with $q \equiv 5 \bmod 8$. Then $\left|S_{\F_q}^{\mathrm{adv}_\infty}\right| = \left|S_{\F_q}^{\mathrm{back}_\infty}\right|$.
\end{corollary}

\begin{remark*}\label{lambda advance backtrack count equality}
    This is basically Corollary Note that Corollary~\ref{advance backtrack count equality} also applies to $k$-values, i.e., $\left|T_{\F_q}^{\mathrm{adv}_\infty}\right| = \left|T_{\F_q}^{\mathrm{back}_\infty}\right|$, as one can deduce from Lemma~\ref{improved biconditional}. In fact, Theorem~\ref{quarter overlap} also applies to $k$-values for the same reason.
\end{remark*}

\section{Reversal symmetry}\label{Reversal symmetry}
As a more convenient way to prove Corollary~\ref{introduction reversal}, we take advantage of birationality between the twice-advanceability elliptic curve $y^2 = 2x\!\left(1+x^2\right)$, where $k=x^2$ in the proof of Theorem~\ref{Legendre elliptic curve counts}, and the twice-backtrackability curve $\mu^4 = 1-k^2$ from Remark~\ref{lambda backtracking criteria}.

In order to prove these theorems pertaining to reversal symmetry, we find it useful to examine the relationship between the curve $y^2=2x(1+x^2)$ from the proof of \ref{Legendre elliptic curve counts} providing square roots $x$ of twice-advanceable $k$-values and the curve $\mu^4 = 1-k^2$ from Remark~\ref{lambda backtracking criteria} providing twice-backtrackable $k$-values.

As the proof of Theorem~\ref{Legendre elliptic curve counts} tells us, any twice-advanceable $k$-value is of the form $x^2$ for some $(x, y)$ satisfying $y^2 = 2x(1+x^2)$; in this case, we have
$$x^2 \overset{\mathrm{k}}{\mapsto} \frac{2x}{1+x^2} = \left(\frac{y}{1+x^2}\right)^2 \overset{\mathrm{k}}{\mapsto} \frac{\pm 2y}{(x+1)^2},$$
thus providing descendants for all twice-advanceable $k$-values and, in finite fields where $-1$ is not a fourth power, providing another method to illuminate which child is indefinitely advanceable. Analogously, Remark~\ref{lambda backtracking criteria} tells us that any twice-backtrackable $k$-value is of the form $k$ for some $(k, \mu)$ satisfying $\mu^4 = 1-k^2$, in which case we have
$$k \overset{\mathrm{k}}{\mapsfrom} \left(\frac{k}{1+\mu^2}\right)^2 = \frac{1-\mu^2}{1+\mu^2} \overset{\mathrm{k}}{\mapsfrom} \left(\frac{1\pm\mu}{1\mp\mu}\right)^2,$$
thus providing ancestors for all twice-backtrackable $k$-values and, in finite fields where $-1$ is not a fourth power, providing another method to illuminate which parent is indefinitely backtrackable.

The existence of analogous chains suggests some relation between the set of twice-advanceable $k$-values and the set of twice-backtrackable $k$-values. We know that $\mu^4 = 1-k^2$ is birational to some elliptic curve, and it turns out to be birational to $y^2=2x(1+x^2)$. Two of the birational maps between them are given by:
\begin{gather*}
    (x, y) = \!\left(\frac{\mu^2}{1+k}, \frac{2\mu}{1+k}\right) \qquad \text{ and } \qquad (k, \mu) = \!\left(\frac{1-x^2}{1+x^2}, \frac{y}{1+x^2}\right) \\
    (x, y) = \!\left(\frac{\mu^2}{1+k}, \frac{-2\mu}{1+k}\right) \qquad \text{ and } \qquad (k, \mu) = \!\left(\frac{1-x^2}{1+x^2}, \frac{-y}{1+x^2}\right)\!.
\end{gather*}
(That these are birational maps between the curves $C_{y^2=2x(1+x^2)}$ and $C_{\mu^4=1-k^2}$ can easily be verified.) Both of these have $k = \frac{1-x^2}{1+x^2}$ (and therefore $x^2 = \frac{1-k}{1+k}$, and in fact, we also have the following commutative diagram in the sense that $\dfrac{1-\frac{2x}{1+x^2}}{1+\frac{2x}{1+x^2}} = \dfrac{1-\mu^2}{1+\mu^2}$ and $\dfrac{1-\frac{2y}{(x+1)^2}}{1+\frac{2y}{(x+1)^2}}, \dfrac{1-\frac{-2y}{(x+1)^2}}{1+\frac{-2y}{(x+1)^2}} \in \left\{\left(\dfrac{1 + \mu}{1 - \mu}\right)^2, \left(\dfrac{1 - \mu}{1 + \mu}\right)^2\right\}$, thus switching $x^2 \overset{\mathrm{k}}{\mapsto} \frac{2x}{1+x^2} \overset{\mathrm{k}}{\mapsto} \frac{\pm 2y}{(x+1)^2}$ and $k \overset{\mathrm{k}}{\mapsfrom} \frac{1-\mu^2}{1+\mu^2} \overset{\mathrm{k}}{\mapsfrom} \left(\frac{1\pm\mu}{1\mp\mu}\right)^2$ by the Möbius transformation $z \mapsto \frac{1-z}{1+z}$, as illustrated in Figure~\ref{fig:chain correspondences}.

\begin{figure}[h!]
    \includegraphics[width=0.95\linewidth]{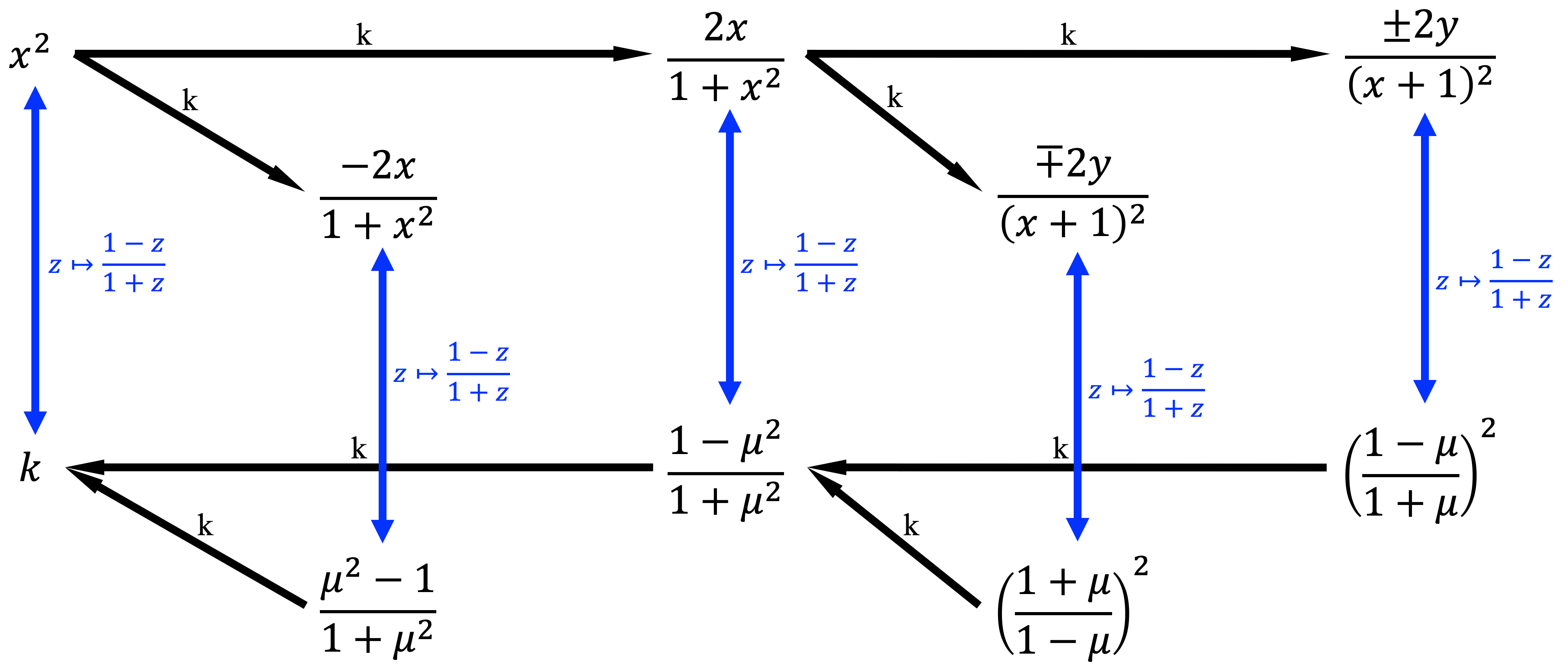}
    \caption{A demonstration of the correspondences between members of $\Sadv{2}{K}$ and $\Sback{2}{K}$ using $(x, y) = \left(\frac{\mu^2}{1+k}, \frac{\pm 2\mu}{1+k}\right)$ and $(k, \mu) = \left(\frac{1-x^2}{1+x^2}, \frac{\pm y}{1+x^2}\right)$}
    \label{fig:chain correspondences}
\end{figure}
Moreover, it is not hard to verify that the two birational maps provided above are the only two birational maps that fit into this diagram, i.e., the only two birational maps such that $z \mapsto \frac{1-z}{1+z}$ (or any Möbius transformation, for that matter) switches $x^2 \overset{\mathrm{k}}{\mapsto} \frac{2x}{1+x^2} \overset{\mathrm{k}}{\mapsto} \frac{2y}{(x+1)^2}$ with $k \overset{\mathrm{k}}{\mapsfrom} \frac{1-\mu^2}{1+\mu^2} \overset{\mathrm{k}}{\mapsfrom} \left(\frac{1\pm\mu}{1\mp\mu}\right)^2$.

This helps us find a contravariant involution $z \mapsto \frac{1-z}{1+z}$ of $T_K$ (contravariant in the sense that it switches advancement and backtracking); in fact, it holds in a very general context independent of previous discussion.

\begin{lemma}\label{self-inverse switch}
    Let $K$ be any field with characteristic not equal to two. The involution $\sigma: k \mapsto \frac{1-k}{1+k}$ of $T_K$ switches advancement and backtracking, i.e., $\sigma\!\left(k_2\right) \overset{\mathrm{k}}{\mapsto} \sigma\!\left(k_1\right)$ if and only if $k_1 \overset{\mathrm{k}}{\mapsto} k_2$.
\end{lemma}

\begin{proof}[Proof of Lemma~\ref{self-inverse switch}]
    That $\sigma$ is an involution of $T_K$ and that $\sigma\!\left(k_2\right) \overset{\mathrm{k}}{\mapsto} \sigma\!\left(k_1\right)$ if and only if $k_1 \overset{\mathrm{k}}{\mapsto} k_2$ can easily be verified. Specifically for the latter,
    $$\sigma\!\left(k_1\right)^2\!\left(\sigma\!\left(k_2\right)+1\right)^2-4\sigma\!\left(k_1\right) = \frac{4}{(1+k_1)^2(1+k_2)^2} \left({k_1}^2\!\left(k_2+1\right)^2-4k_1\right).\qedhere$$
\end{proof}

Applying the contravariant involution $\sigma$ to $G_K$ and subgraphs of $G_K$ produces contravariant isomorphisms between various subgraphs.

\begin{corollary}\label{lambda-reversal}
    For any field $K$ with characteristic not equal to two, the graph $G_K$ is isomorphic to its own reversal, and for any $0\leq n\leq \infty$, the graph $G_K^{\mathrm{adv}_n}$ is isomorphic to the reversal of $G_K^{\mathrm{back}_n}$.
\end{corollary}

\begin{proof}[Proof of Corollary~\ref{lambda-reversal}]
    We know from Lemma~\ref{self-inverse switch} that $\sigma$ is a contravariant involution of $G_K$. As a result, $\sigma$ restricts to a contravariant isomorphism between $G_K^{\mathrm{adv}_n}$ and $G_K^{\mathrm{back}_n}$.
\end{proof}

Ideally, we would like to prove the general existence of an analogous contravariant isomorphism for nodes. However, as Remark~\ref{no sigma-lifting} will show, any such contravariant isomorphism would not be compatible with $\sigma$. However, Corollary~\ref{lambda-reversal} does allow us prove the similarities between $F_K^{\mathrm{adv}_\infty}$ and $F_K^{\mathrm{back}_\infty}$ outlined in Theorem~\ref{population reversal main theorem}, such as that the maximum tentacle length equals the maximum colon length, as a sign that reversal symmetry may also apply for nodes. This will first be proven for $k$-values then extended to nodes using Corollary~\ref{improved biconditional}.

\begin{corollary}\label{T correspondences}
    If $K$ is any field with characteristic not equal to two, then
    \begin{enumerate}
        \item For $n \geq 0$, $T_K^{\mathrm{adv}_n}=T_K^{\mathrm{adv}_\infty}$ if and only if $T_K^{\mathrm{back}_n}=T_K^{\mathrm{back}_\infty}$.
        \item For $n \geq 0$, $\Tadv{n}{\Fq} \cap \Tback{\infty}{\Fq} = T_{\Fq}^\mathrm{cyc}$ if and only if $\Tback{n}{\Fq} \cap \Tadv{\infty}{\Fq} = T_{\Fq}^\mathrm{cyc}$.
        \item For $0 \leq n \leq \infty$, $\abs{T_K^{\mathrm{adv}_n}}=\abs{T_K^{\mathrm{back}_n}}$, including $\abs{T_K^{\mathrm{adv}_n}}$ being finite if and only if $\abs{T_K^{\mathrm{back}_n}}$ is finite.
    \end{enumerate}
\end{corollary}

\begin{proof}[Proof of Corollary \ref{T correspondences}]
    All parts follow from the aforementioned involution $\sigma$ on $T_K$ from Corollary \ref{lambda-reversal} that switches $k$-advancement and $k$-backtracking. For the first part, if $T_K^{\mathrm{adv}_n}=T_K^{\mathrm{adv}_\infty}$, then $T_K^{\mathrm{back}_n} = \sigma\!\left(T_K^{\mathrm{adv}_n}\right) = \sigma\!\left(T_K^{\mathrm{adv}_\infty}\right) = T_K^{\mathrm{back}_\infty}$; the proof of the converse is analogous.
    
    For the second part, if $T_K^{\mathrm{adv}_n} \cap T_K^{\mathrm{back}_\infty}=T_{\Fq}^\mathrm{cyc}$, then $T_K^{\mathrm{back}_n} \cap T_K^{\mathrm{adv}_\infty} = \sigma\!\left(T_K^{\mathrm{adv}_n} \cap T_K^{\mathrm{back}_\infty}\right) = \sigma\!\left(T_{\Fq}^\mathrm{cyc}\right) = T_{\Fq}^\mathrm{cyc}$; the proof of the converse is analogous.
    
    The third part follows easily from the existence of a bijection between $T_K^{\mathrm{adv}_n}$ and $T_K^{\mathrm{back}_n}$.
\end{proof}

\begin{corollary}\label{S correspondences}
    If $K$ is any field with characteristic not equal to two, then
    \begin{enumerate}
        \item For $n \geq 0$, $\Sadv{n}{K}=\Sadv{\infty}{K}$ if and only if $\Sback{n}{K}=\Sback{\infty}{K}.$
        \item For $n \geq 0$, $\Sadv{n}{K} \cap \Sback{\infty}{K} = S_K^\mathrm{cyc}$ if and only if $\Sback{n}{K} \cap \Sadv{\infty}{K} = S_K^\mathrm{cyc}$.
        \item For $0 \leq n \leq \infty$, $\abs{\Sadv{n}{K}}=\abs{\Sback{n}{K}}$, including $\abs{S_K^{\mathrm{adv}_n}}$ being finite if and only if $\abs{S_K^{\mathrm{back}_n}}$ is finite.
    \end{enumerate}
\end{corollary}

\begin{proof}[Proof of Corollary~\ref{S correspondences}]
    Each part follows from the respective part of Corollary~\ref{T correspondences} because (as Corollary~\ref{improved biconditional} assures us) $S_K^{\mathrm{adv}_n} = \{(a, ka): a \in K^\times, k \in T_K^{\mathrm{adv}_n}\}$ and $S_K^{\mathrm{back}_n} = \{(a, ka): a \in K^\times, k \in T_K^{\mathrm{back}_n}\}$ for $0 \leq n \leq \infty$.
\end{proof}

\section{Jellyfish structure}\label{Jellyfish structure}
Now, we shall figure out the structure of the jellyfish in order to prove part (5) of Theorems~\ref{5mod8results}~and~\ref{5mod8backtrackingresults}, and well as proving Corollary~\ref{introduction reversal}.

\begin{theorem}\label{tentacle colon structure}
    Let $q$ be a prime power with $q \equiv 5 \bmod 8$. The graph $F_{\F_q}^{\mathrm{adv}_\infty}$ consists of bell-head (i.e., simple) cycles where each node in the cycle has one tentacle, which is Y-shaped and has length two, and the graph $F_{\F_q}^{\mathrm{back}_\infty}$ consists of bell-head cycles where each node in the cycle has one colon, which is Y-shaped and has length two.
\end{theorem}

\begin{proof}[Proof of Theorem~\ref{tentacle colon structure}]
    Corollary~\ref{SFqadm formula} and Lemma~\ref{backtrackability extension} tell us that advancement and backtracking are single-valued on $S_{\Fq}^\mathrm{cyc}$, so all cycles are simple (i.e., have no branching), and $S_{\Fq}^\mathrm{cyc}$ is their disjoint union.
    
    To prove the statement about tentacles in $F_{\F_q}^{\mathrm{adv}_\infty}$, note that any node in $S_{\Fq}^\mathrm{cyc}$ is in $S_{\F_q}^{\mathrm{back}_2}$, so Corollary~\ref{parental backtrackability} and Lemma~\ref{backtrackability extension} tell us that of its two parents, one is in $S_{\F_q}^{\mathrm{back}_2}$ and therefore in $S_{\F_q}^{\mathrm{cyc}}$ (because it is already known to be in $S_{\F_q}^{\mathrm{adv}_\infty}$ and $S_{\F_q}^{\mathrm{back}_\infty}=S_{\F_q}^{\mathrm{back}_2}$), and the other is $S_{\F_q}^{\mathrm{back}_1} \setminus S_{\F_q}^{\mathrm{cyc}}$ and has two parents, neither of which is backtrackable, thus forming a Y-shaped tentacle of length two. Overlap between different cyclic nodes' tentacles is ruled out because advancement on $S_{\F_q}^{\mathrm{adv}_\infty}$ is single-valued by \ref{SFqadm formula}.
    
    To prove the statement about colons in $F_{\F_q}^{\mathrm{back}_\infty}$, note that any node in $S_{\Fq}^\mathrm{cyc}$ is in $S_{\F_q}^{\mathrm{adv}_2}$, so Lemmas \ref{fourth power lemma}~and~\ref{advanceability extension} tell us that of its two children, one is in $S_{\F_q}^{\mathrm{adv}_2}$ and therefore in $S_{\F_q}^{\mathrm{cyc}}$ (because it is already known to be in $S_{\F_q}^{\mathrm{back}_\infty}$ and $S_{\F_q}^{\mathrm{adv}_\infty}=S_{\F_q}^{\mathrm{adv}_2}$), and the other is $S_{\F_q}^{\mathrm{adv}_1} \setminus S_{\F_q}^{\mathrm{cyc}}$ and has two children, neither of which is advanceble, thus forming a Y-shaped colon of length two. Overlap between different cyclic nodes' colons is ruled out because backtracking on $S_{\F_q}^{\mathrm{back}_\infty}$ is single-valued by \ref{backtrackability extension}.

    Overlap between a tentacle and a colon, between one jellyfish's tentacle and another jellyfish's cycle, or between one jellyfish's colon and another jellyfish's cycle is ruled out by the fact that cyclic nodes are both indefinitely advanceable and indefinitely backtrackable, tentacle nodes are indefinitely advanceable but not indefinitely backtrackable, and colon nodes are indefinitely backtrackable but not indefinitely advanceable.
\end{proof}

This allows us to prove the analogue of Corollary~\ref{lambda-reversal} for the nodes, as was speculated about in the discussion immediately after its proof, for odd $q \not\equiv 1 \bmod 8$.

\begin{corollary}\label{node-reversal}
    Let $q$ be an odd prime power not congruent to 1 modulo 8. Then $F_{\Fq}^{\mathrm{adv}_\infty}$ is isomorphic to the reversal of (i.e., contravariantly isomorphic to) $F_{\Fq}^{\mathrm{back}_\infty}$.
\end{corollary}

\begin{proof}[Proof of Corollary~\ref{node-reversal}]
    The cases of $\qthreemodfour$ and $\qfivemodeight$ shall be proven separately. For the case of $\qfivemodeight$, the result follows easily from Theorem~\ref{tentacle colon structure} because of the shared structure between tentacles and colons. Specifically, Theorem~\ref{tentacle colon structure} tells us that the cycles in $F_{\Fq}^{\mathrm{adv}_\infty}$ and $F_{\Fq}^{\mathrm{back}_\infty}$ are all simple, so $F_{\Fq}^\mathrm{cyc}$ must consist of simple cycles (no branching) and therefore must be isomorphic to its own reversal. Because we know from Corollary~\ref{improved biconditional} and Lemma~\ref{lambda-reversal} (or, since $\qfivemodeight$, from Theorem~\ref{tentacle colon structure}) that tentacles and colons have analogous shapes, so because $F_{\Fq}^\mathrm{cyc}$ is isomorphic to its own reversal, the desired result follows.

    The case of $\qthreemodfour$ is analogous, again because of the shared structure between tentacles and colons. Theorem~1 part (3) of \cite{MR4567422} tells us that the cycles in $F_{\Fq}^{\mathrm{adv}_\infty}$, i.e., $F_{\Fq}^\mathrm{cyc}$, must consist of bell-head cycles (i.e., no branching), and so $F_K^\mathrm{cyc}$ must be isomorphic to its own reversal. Theorem~1 part (3) of \cite{MR4567422} also tells us that each cycle node is pointed to by one tentacle, which has length one and no branching. Furthermore, as pages 1 and 2 of \cite{MR4567422} explain, each node in $S_{\Fq}^{\mathrm{adv}_1}$, including $F_{\Fq}^\mathrm{cyc}$, has two children, of which exactly one is in $S_{\Fq}^{\mathrm{adv}_1}$. For parents in $F_{\Fq}^\mathrm{cyc}$, this other child is a length-one colon. Thus, when $\qthreemodfour$, tentacles and colons have analogous shapes (which could alternatively could have been deduced from Corollary~\ref{improved biconditional} and Lemma~\ref{lambda-reversal}), so because $F_{\Fq}^\mathrm{cyc}$ is isomorphic to its own reversal, the desired result follows.
\end{proof}

\begin{remark}\label{no sigma-lifting}
    Here, we explain why the contravariant isomorphism for nodes guaranteed by Corollary~\ref{node-reversal} cannot be compatible with the involution $\sigma$ in Lemma~\ref{self-inverse switch} for $k$-values. The skeletons of $F_{\F_{29}}^{\mathrm{adv}_\infty}$ and $F_{\F_{29}}^{\mathrm{back}_\infty}$ (namely, their graph structures without node labels) are shown in Figure~\ref{fig:F29skeletons earlier copy}, and note that they are in contravariant isomorphism.
    \begin{figure}[h]
        \begin{subfigure}{0.3\textwidth}
            \includegraphics[height=\linewidth, angle=90]{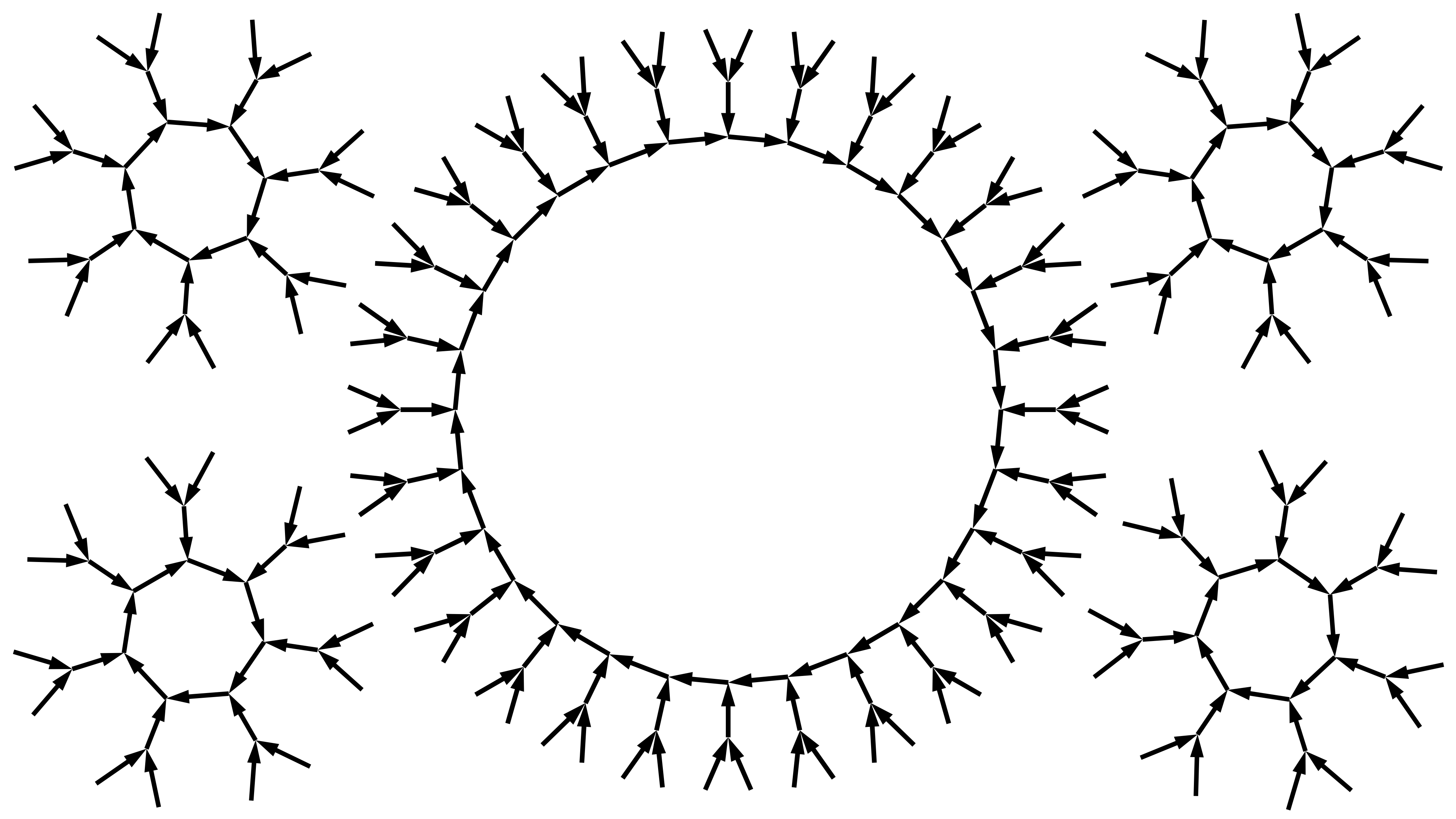} 
            \caption{$F_{\F_{29}}^{\mathrm{adv}_\infty}$ without node labels}
        \end{subfigure}
        \hspace{0.166\textwidth}
        \begin{subfigure}{0.3\textwidth}
            \includegraphics[height=\linewidth, angle=90]{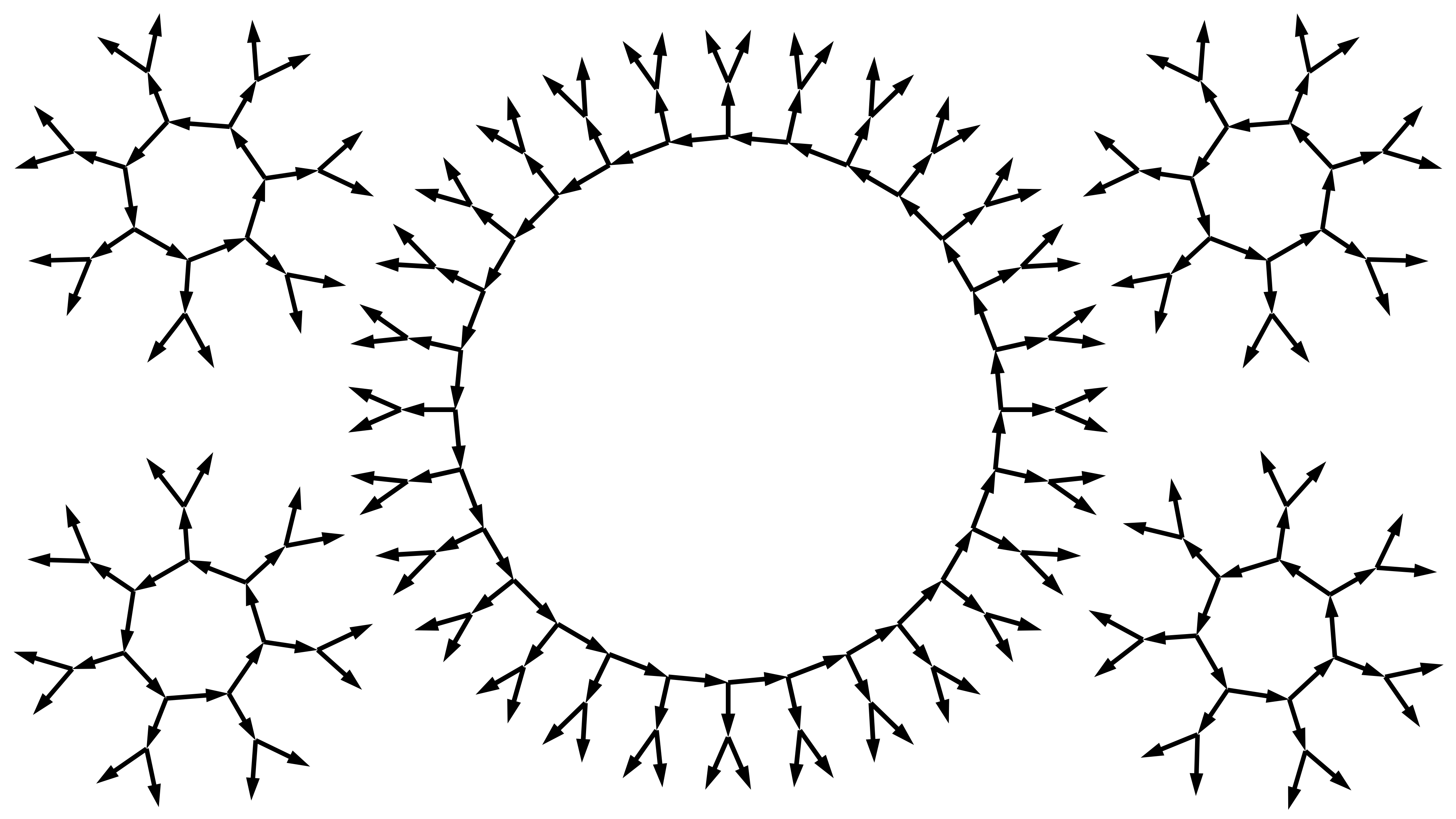}
            \caption{$F_{\F_{29}}^{\mathrm{back}_\infty}$ without node labels}
        \end{subfigure}
        \caption{Skeletons of $F_{\F_{29}}^{\mathrm{adv}_\infty}$ and $F_{\F_{29}}^{\mathrm{back}_\infty}$ shown side-by-side}
        \label{fig:F29skeletons earlier copy}
    \end{figure}
    \begin{figure}[h]
        \begin{subfigure}{0.495\textwidth}
            \includegraphics[width=\linewidth]{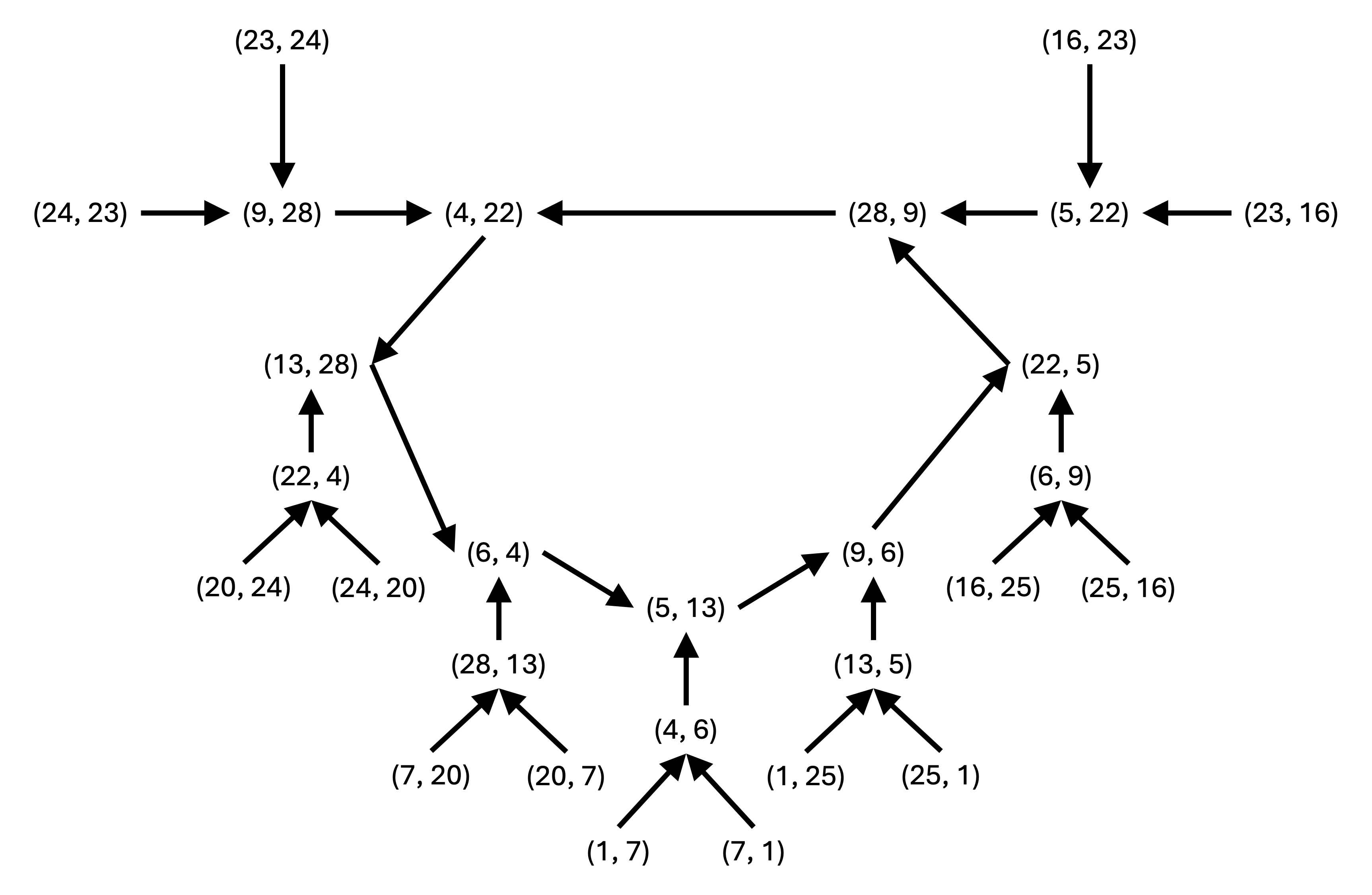} 
            \caption{One connected component of $\Sadv{\infty}{\F_{29}}$}
            \label{fig:forwardscomparisononleft}
        \end{subfigure}
        \hfill
        \begin{subfigure}{0.495\textwidth}
            \includegraphics[width=\linewidth]{Backwards_F29_jellyfish.png}
            \caption{One connected component of $\Sback{\infty}{\F_{29}}$}
            \label{fig:backwardscomparisononright}
        \end{subfigure}
        \caption{One connected component each of $\Sadv{\infty}{\F_{29}}$ and $\Sback{\infty}{\F_{29}}$, showing contravariantly isomorphic structure}
        \label{fig:F29sidebyside earlier copy}
    \end{figure}
    For more details, one of the four small jellyfish are shown with labels in Figure~\ref{fig:F29sidebyside earlier copy}. We point out that every node in the long cycle of length 28 has $k$-value $\overline{6}$, and every node in the four short cycles of length 7 has $k$-value $\overline{20}$. However, $\sigma\!\left(\overline{6}\right)=\overline{20}$. Thus, any contravariant isomorphism $\tau:F_{\F_{29}}^{\mathrm{adv}_\infty}\to F_{\F_{29}}^{\mathrm{back}_\infty}$ cannot not be a lift of $\sigma$ because $\tau$ must preserve the cycle length.
\end{remark}

For a prime power $q\equiv 1\bmod 8$, we speculate that a reversal symmetry between $F_{\F_{q}}^{\mathrm{adv}_\infty}$ and $F_{\F_{q}}^{\mathrm{back}_\infty}$ exists, but any advancement in this direction seems to require knowledge specific to $q \equiv 1 \bmod 8$ of the precise structure of both $F_{\F_{q}}^{\mathrm{adv}_\infty}$ and $F_{\F_{q}}^{\mathrm{back}_\infty}$. We conclude this section by displaying the structure of $G_{\F_{113}}^{\mathrm{adv}_\infty} \cup G_{\F_{113}}^{\mathrm{back}_\infty}$ (in Figure~\ref{fig:F113sidebyside}), which shows that advancement and backtracking are not single-valued.

\begin{figure}[h]
    \floatbox[{\capbeside\thisfloatsetup{capbesideposition={right,center},capbesidewidth=0.65\textwidth}}]{figure}[\FBwidth]{\caption{An image of $G_{\F_{113}}^{\mathrm{adv}_\infty} \cup G_{\F_{113}}^{\mathrm{back}_\infty}$ (generated by Wolfram Mathematica) with $G_{\F_{113}}^\mathrm{cyc}$ in black, $G_{\F_{113}}^{\mathrm{adv}_\infty} \setminus G_{\F_{113}}^\mathrm{cyc}$ in green, and $G_{\F_{113}}^{\mathrm{back}_\infty} \setminus G_{\F_{113}}^\mathrm{cyc}$ in red, showing that $\overline{9}$, $\overline{88}$, $\overline{26}$, and $\overline{100}$ can each be advanced in $G_{\F_{113}}^{\mathrm{adv}_\infty}$ in multiple ways and that $\overline{46}$, $\overline{67}$, $\overline{20}$, and $\overline{93}$ can each be backtracked in $G_{\F_{113}}^{\mathrm{back}_\infty}$ in multiple ways}\label{fig:F113sidebyside}}{\includegraphics[width=0.35\textwidth]{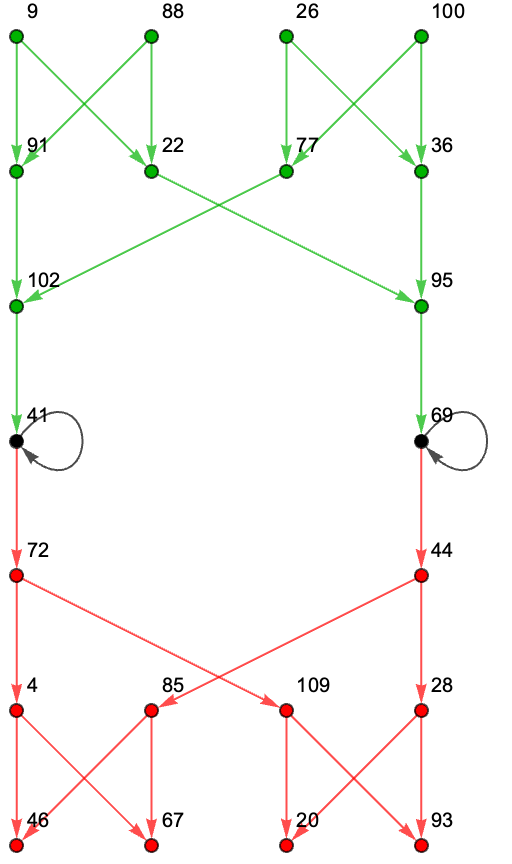}}
\end{figure}

\section{Proof of theorems}\label{Proof of theorems}
Here, we relate the theorems in the introduction to statements we have proven.

\subsection{Proof of Theorem~\ref{5mod8results}}
\begin{proof}
    Part (1), part (4), and the first part of part (2) of Theorem~\ref{5mod8results} come from Corollary~\ref{SFqadm formula} because 16 is a fourth power. The rest of part (2) of Theorem~\ref{5mod8results} comes from Corollary~\ref{nonempty past 13}. Part (3) of Theorem~\ref{5mod8results} comes from Corollary~\ref{SFqadvinfty population}. Part (5) of Theorem~\ref{5mod8results} comes from Theorem~\ref{tentacle colon structure}.
\end{proof}

\subsection{Proof of Corollary~\ref{introduction reversal}}
\begin{proof}
    The theorem follows from Corollary~\ref{node-reversal}.
\end{proof}

\subsection{Proof of Theorem~\ref{5mod8backtrackingresults}}
\begin{proof}
    Part (1) and the first part of part (2) of Theorem~\ref{5mod8backtrackingresults} come from Corollary~\ref{backtracking twice to infinite}. The rest of part (2) and all of part (3) of Theorem~\ref{5mod8backtrackingresults} can be deduced from part (2) of Theorem~\ref{5mod8results} and Corollary~\ref{advance backtrack count equality}. Part (4) of Theorem~\ref{5mod8backtrackingresults} comes from Lemma~\ref{backtrackability extension}. Part (5) of Theorem~\ref{5mod8results} comes from Theorem~\ref{tentacle colon structure}.
\end{proof}

\subsection{Proof of Theorem~\ref{population reversal main theorem}}
\begin{proof}
    All three parts of the theorem follow from Corollary~\ref{S correspondences}.
\end{proof}

\subsection{Proof of Proposition~\ref{cycle fractions}}
\begin{proof}
    The result follows from Theorem~\ref{quarter overlap}.
\end{proof}

\section*{Acknowledgements}
I would like to thank Professors Yifeng Huang, Dagan Karp, and Ken Ono for their guidance with this paper. In addition, this work is a continuation of work that I did the previous summer funded by the Giovanni Borrelli Fellowship from Harvey Mudd College, so I would like to thank the college and its donors for funding this research and Professor Lenny Fukshansky for advising me then. I was first introduced to studying the arithmetic-geometric progression over finite fields by a talk by Eleanor McSpirit at the UVA REU in Number Theory run by Professor Ono, so I would like to thank her for introducing me to this topic and to everyone who sponsored or helped organize this program. I would also like to thank one of my REU researchmates, Noah Walsh, for helping me bound the population of the $q \equiv 5 \bmod 8$ jellyfish swarms more closely. Lastly, I would like to acknowledge Évariste Galois for developing many interesting and useful mathematical results that my work on this topic uses; while this last type of acknowledgment is admittedly not common, maybe it should be.

\end{document}